\documentclass[12pt]{amsart}
\usepackage{xcolor}
\usepackage{mathtools}
\usepackage[latin2]{inputenc}
\usepackage{amsmath, amssymb}

\usepackage[
  hmarginratio={1:1},     % equal left and right margins
  vmarginratio={1:1},     % equal top and bottom margins
  textwidth=430pt,        % new text width
  heightrounded,          % always useful
]{geometry}

\usepackage{hyperref}
\usepackage{enumerate}
\usepackage{float}

\newtheorem{theorem}{Theorem}
\newtheorem{lemma}{Lemma}
\newtheorem{corollary}{Corollary}
\newtheorem{proposition}{Proposition}
\newtheorem{algorithm}{Algorithm}
\theoremstyle{definition}
\newtheorem{definition}{Definition}
\theoremstyle{remark}
\newtheorem{remark}{Remark}

\newcommand{\divv}{\operatorname{div}}
\newcommand{\volm}{\operatorname{Vol}_{M}}

\allowdisplaybreaks

\usepackage{xcolor}

\title[Large data limit of the MBO scheme]{Large data limit of the MBO scheme for data clustering: $\Gamma$-convergence of the thresholding energies}

\author{Tim Laux}
\address[Tim Laux]{Institut f\"ur angewandte Mathematik, Universit\"at Bonn, Endenicher Allee 60, 53115 Bonn, Germany}
\email{tim.laux@hcm.uni-bonn.de}

\author{Jona Lelmi}
\address[Jona Lelmi]{Institut f\"ur angewandte Mathematik, Universit\"at Bonn, Endenicher Allee 60, 53115 Bonn, Germany}
\email{lelmi@hcm.uni-bonn.de}

\begin{document}

\maketitle

\begin{abstract}
In this work we present the first rigorous analysis of the MBO scheme for data clustering in the large data limit. Each iteration of the scheme corresponds to one step of implicit gradient descent for the thresholding energy on the similarity graph of some dataset. For a subset of the nodes of the graph, the thresholding energy at time $h$ measures the amount of heat transferred from the subset to its complement at time $h$, rescaled by a factor $\sqrt{h}$. It is then natural to think that outcomes of the MBO scheme are (local) minimizers of this energy. We prove that the algorithm is consistent, in the sense that these (local) minimizers converge to (local) minimizers of a suitably weighted optimal partition problem.

\medskip


  \medskip

\noindent \textbf{Mathematical Subject Classification (MSC2020)}:
	35A15 (primary), % Variational methods applied to PDEs
	53Z50, % Applications of differential geometry to data and computer science 
	49Q20, % Variational problems in a geometric measure- theoretic setting
	35R01, % PDEs on manifolds [See also 32Wxx, 53Cxx, 58Jxx]
	35R02. % PDEs on graphs and networks (ramified or polygonal spaces) 
\end{abstract}

\section{Introduction}\label{sec:intro}

The MBO scheme was introduced by Merriman, Bence and Osher in~\cite{Merriman1992},~\cite{Merriman1994} as an efficient algorithm to approximate mean curvature flow. It was subsequently adapted to perform data clustering by Bertozzi et al.~\cite{Merkurjev2013},~\cite{Merkurjev2014},~\cite{Gennip2014}. In the current work we present the first rigorous analysis of the large data limit of the algorithm. In the simplest setting, data clustering aims at partitioning some data points $\{X_1, . . . , X_n\} \subset \mathbf{R}^d$ into two meaningful clusters. One constructs the \textit{similarity graph} $G=(V, E)$ where the vertex set $V$ is given by the data points, and the edge between two vertices is weighted by a function of their Euclidean distance. One seeks a function $\chi: V \to \{0,1\}$ encoding the desired partition. Denoting by $\Delta$ a suitable graph Laplacian on $G$, the MBO scheme works as follows. Let $h > 0$ be a given time-step size. Let $\chi: V \to \{0,1\}$ be an initial guess for the clustering. To obtain a new clustering one performs $N$ times the following steps.
\begin{enumerate}
\item \textbf{Diffusion}. Define
\begin{equation}
u(t) = e^{-t\Delta}\chi.\nonumber
\end{equation}
\item \textbf{Thresholding}. Update the cluster by setting
\begin{equation}
\chi = 1  \Leftrightarrow u(h) \ge \frac{1}{2}.\nonumber
\end{equation}
\end{enumerate}
The output clustering is given by $\chi$. From the point of view of applications, it is important to be able to include constraints on some of the nodes of the graph. More precisely, frequently some of the labels of the data points are already known - in the sense that it is known to which of the clusters these points belong. The aim then is to classify all the other data points using both the known labels and the geometry of the dataset. The MBO scheme can be suitably modified to take into account given labels: It suffices to choose an appropriate forcing function $f:V \to \mathbf{R}$ and change the thresholding value in the second step of the algorithm to $\frac{1}{2} - \sqrt{h}f$. Some of the possible choices for $f$ are discussed in Section~\ref{sec:semisupervised}. A crucial observation is that each iteration of the MBO scheme decreases the energy

\begin{equation*}
F_h(u) = \frac{1}{\sqrt{h}} \langle 1-u, e^{-h\Delta}u\rangle - \langle f, u \rangle,
\end{equation*}
the first term of which we call the thresholding energy. This was observed in this setting by Bertozzi et al.~\cite{Gennip2014}, but was originally pointed out by Esedo\={g}lu and Otto in their seminal work~\cite{Esedoglu2015}, where they give a minimizing movements interpretation for the continuum MBO scheme. Each step of the MBO scheme corresponds thus to one step of implicit gradient descent for the energy $F_h$ and if $N$ is large enough we may think of the outputs of the MBO scheme as almost  (local) minimizers of the energy $F_h$. It is then natural to raise the following question: Can we think of these local minimizers as approximations of a minimization problem of some energy defined in the continuum? The right tool to tackle this problem is $\Gamma$-convergence, a notion of convergence that ensures convergence of (local) minimizers. This was introduced by De Giorgi and Franzoni in~\cite{DeGiorgi1975}; see~\cite{DalMaso1993} and~\cite{Braides2002} for background on the topic. Whilst the case $f \neq 0$ is the one which is used for applications, the additional term is not the most relevant one in the energy $F_h$: Indeed, once one is able to identify the correct $\Gamma$-limit for the first term --- the \textit{thresholding energy} --- the convergence of the full energy $F_h$ is just a corollary obtained by observing that the second term is just a continuous perturbation, which is stable with respect to $\Gamma$-convergence; see Section~\ref{sec:semisupervised} for the details. Thus the main point is to study the asymptotic behavior of the thresholding energy.

More precisely, as usual in data science, we work under the \textit{manifold assumption}, i.e., we assume that the data points lay on a $k$-dimensional Riemannian submanifold $M \subset \mathbf{R}^d$, and that these are independently distributed according to a probability measure $\nu = \rho\volm$ which is absolutely continuous with respect to the volume measure with a smooth and positive density. We want to address the following question: What is the limiting minimization problem when the number of data points goes to infinity and the step size $h$ goes to zero? It turns out that the first limit --- letting the number of data points go to infinity --- is the continuum thresholding energy, which is defined as follows:

\begin{equation*}
E_h(u) = \frac{1}{\sqrt{h}} \int_M (1-u) e^{-h\Delta_{\rho^2}}u \rho^2d\volm,\quad u: M \to [0,1]\ \text{measurable}.
\end{equation*}
Here $e^{-h\Delta_{\rho^2}}$ is the heat operator on the manifold associated to the weighted Laplacian with weight $\mu := \rho^2 \volm$; see Section~\ref{sec:prelim} for the details. Of course, it is important to specify the topology under which this convergence takes place: at a first sight, this may seem subtle because there is no way to define the restriction of a measurable function on the manifold to a discrete set of points (like the vertex set of a finite graph). To overcome this difficulty, in their seminal paper~\cite{GarciaTrillos2016}, Garc\'{i}a Trillos and Slep\v{c}ev introduced the notion of $TL^p$-convergence --- a notion of convergence based on optimal transport allowing one to compare functions which are in $L^p$-spaces with respect to two different probability measures. This notion of convergence has become a standard tool in studying large data limits, a non exhaustive list of relevant works being~\cite{GarciaTrillos2018, Slepcev2019, Dunlop2020}. In our context, however, we can even show that this $\Gamma$-convergence takes place in a \textit{weak $TL^2$-topology}, see Definition~\ref{def:weaktlp}. This is the natural convergence to work with, because the weak compactness properties of the thresholding energies only yield $\Gamma$-compactness in the weak $TL^2$-topology.

The second limit --- letting $h \downarrow 0$ --- is understood in the sense of $\Gamma$-convergence with respect to the strong $L^1(M)$-topology. This problem is well studied when $M = \mathbf{R}^k$ is the flat Euclidean space and $\rho = 1$ is the constant unit density. In particular, in~\cite{Miranda2007a} Miranda, Pallara, Paronetto and Preunkert prove the \textit{consistency} of the thresholding energy in the flat space with unit density, i.e., they show that if $u$ is the characteristic function of a set of finite perimiter, $E_h(u)$ approximates the perimeter as $h \downarrow 0$. The full $\Gamma$-convergence in the flat case with constant unit density is a consequence of the work of Alberti and Bellettini~\cite{Alberti1998} and is based on the blow-up method of Fonseca and M\"{u}ller~\cite{Fonseca1993}. The $\Gamma$-convergence in the multiphase setting --- still in the flat case with constant unit density --- is proved in~\cite{Esedoglu2015} based on a monotonicity formula. On a curved space, the only analogous result~\cite{Miranda2007} concerns a different --- and simpler --- energy $\tilde{E}_h(u) = \int_M |\nabla e^{-h\Delta} u |\, d\volm$. In particular neither the $\Gamma$-convergence nor the consistency for the thresholding energy was known. We prove both these results in the vectorial multiclass setting. 

Let us briefly comment on the key ingredients for proving the main results in this paper, namely the discrete-to-nonlocal convergence of Theorem~\ref{thm:discretenonloc} and the nonlocal-to-local convergence of Theorem~\ref{thm:nonlocloc}. The main point for proving that the thresholding energies on the graph converge to the nonlocal thresholding energy on the manifold as the number of data points goes to infinity is the convergence and regularization properties of the heat operators: we prove that whenever $u_n$ are functions defined on the first $n$ data points which converge to $u \in L^2(M)$ weakly with respect to the $TL^2(M)$-convergence, then for each fixed $h > 0$ the corresponding heat operators $e^{-h\Delta}u_n$ converge (up to constants in the weighted Laplacian) to $e^{-h\Delta_{\rho^2}}u$ \textit{strongly} in $TL^2(M)$. The convergence of the thresholding energies is then just a corollary that uses the fact that products of weakly convergent functions and strongly convergent ones converge to the product of their limits.  
With a similar argument we also give a positive answer to a question raised by Bertozzi et al.~\cite[Question 7.5]{Gennip2014} regarding the consistency of the MBO scheme on random geometric graphs with a fixed time-step size, see Corollary~\ref{cor:bertozzi_qst}. 
To prove the \textit{weak-to-strong} convergence for the heat operators, one reduces to the case when $u_n = u_{|_{\{X_1, . . . , X_n\}}}$ is the restriction of a \textit{smooth} function $u \in C^{\infty}(M)$ to the first $n$ data points. This is then handled as follows: Like any other gradient flow, solutions $v_n(t)$ to the heat equation on the graph $G_n$ of the first $n$ data points with initial datum $v_n(0) = u_n$ satisfy an optimal energy dissipation inequality, i.e.,

\begin{equation*}
E_n[v_n(t)] + \frac{1}{2}\int_0^t |\Delta_{n}v_n(s)|^2_{\mathcal{V}_{n}}ds+ \frac{1}{2}\int_0^t |\frac{d}{ds}v(s)|^2_{\mathcal{V}_{n}}ds \le E_n[u_n],
\end{equation*}
where $E_n$ is the Dirichlet energy on the graph and $| \cdot |_{\mathcal{V}_n}$ is a norm induced by a suitable scalar product so that the Laplacian is self-adjoint. Garc\'{i}a Trillos and Slep\u{c}ev~\cite{GarciaTrillos2018} proved that the energies $E_n$ $\Gamma$-converge in the $TL^2$-sense to the Dirichlet energy in the continuum. Using this result and the lower semicontinuity of the other left hand side terms one can pass to the limit as $n \to +\infty$ in the previous inequality to obtain that $v_n$ converges strongly to a function $v$ in $TL^2(M)$ that satisfies the following inequality:

\begin{align*}
\begin{aligned}
E[v(t)] + \frac{1}{2}\int_0^t \int_{M} |\Delta_{\rho^2} v|^2\rho^2 d\volm ds + \frac{1}{2}\int_0^t \int_{M} |\frac{dv}{ds}|^2\rho^2 d\volm ds \le E[u],
\end{aligned}
\end{align*}
where $E$ is the Dirichlet energy in the continuum. This is nothing but the optimal energy dissipation inequality associated to the $L^2(M)$-gradient flow for the energy $E$, thus it implies that $v(t) = e^{-t\Delta_{\rho^2}}u$.

The second limit, i.e., letting the step size $h$ go to zero, is the most technical and challenging part of our work. The main difficulty is that the heat kernel on a manifold does not enjoy any translation invariance property as the standard Euclidean heat kernel. To overcome this problem, we use a careful localization argument in space-time which allows to approximate the heat kernel on the manifold with the standard heat kernel on the tangent bundle - this is achieved by means of the asymptotic expansion for the heat kernel, c.f. Section~\ref{sec:prelim}. The $\Gamma$-$\limsup$ is achieved by localizing on the reduced boundary of the set of finite perimeter, while the $\Gamma$-$\liminf$ inequality is obtained by means of the blow-up method of Fonseca and M\"{u}ller~\cite{Fonseca1993}, see also~\cite{Alberti1998, Ambrosio2011}.

As we already pointed out earlier, Esedo\={g}lu and Otto gave a minimizing movements interpretation for the MBO scheme in the continuum, which still holds true at the discrete level. When talking about minimizing movements, it is important to specify the dissipation mechanism, i.e., the metric on the space where one constructs the minimizing movements. This interpretation allows one to prove the convergence of the multiphase MBO scheme to suitably defined notions of weak solutions of mean curvature flow, namely BV solutions~\cite{Laux2016}, Brakke's solutions~\cite{Laux2020} and De Giorgi's solutions~\cite{Laux2019lecn, Laux2021}. Previously, the convergence of the MBO scheme in the simple two phase setting was obtained independently in~\cite{Evans1993} and~\cite{Barles1995}. In the latter works, weak solutions are defined in terms of viscosity solutions --- but this notion cannot be extended to the multiphase setting due to the lack of a maximum principle. In the present work it is not important to understand the dissipation mechanism, i.e., it is not important to specify the metric of the minimizing movements, because we just focus on the behavior of (local) minimizers of the energies $F_h$. As the selection of local minimizers strongly depends on the evolution, we will address the problem of the convergence of the dynamics in the follow-up work~\cite{LauxLelmiFuture}. There, we will present a different analysis of the scheme where we exploit the comparison principle to prove convergence to the viscosity solution of mean curvature flow.

Finally, let us point out that we decided to carry out our analysis for the random walk Laplacian. There is no particular reason for this choice, and our results may be suitably modified to obtain similar statements for other choices of the Laplacian --- like the unormalized one. One may even allow for data dependent weights. We summarize these results in Section~\ref{sec:discussions}.

\medskip

The rest of the paper is organized as follows. In Section~\ref{sec:thescheme}, we recall some notation for graphs, we recall the definition of the random walk Laplacian and we recall the multiclass MBO scheme. In Section~\ref{sec:mainres}, we state the main results of this paper. Section~\ref{sec:semisupervised} contains a discussion of the semi-supervised MBO scheme and an explanation of the straight-forward extension of the results of Section~\ref{sec:mainres}  to the semi-supervised setting. Section~\ref{sec:discussions} contains a discussion of the missing ingredient needed for obtaining the joint limit $n \to +\infty$ and $h \downarrow 0$ --- namely the monotonicity of the thresholding energy. This section also contains a discussion on how our results may be extended to other graph Laplacians --- in particular with data dependent weights. Section~\ref{sec:prelim} contains preliminaries about weighted manifolds, the $TL^p$-convergence and other well-known material such as the $\Gamma$-convergence of the Dirichlet energies in the $TL^2$-topology. Section~\ref{sec:proofs} contains the proofs of the results stated in Section~\ref{sec:mainres}. Finally, the \hyperlink{sec:appendix}{Appendix} contains proofs of some known results which are needed but that we were not able to trace back in the literature.

\section{The MBO scheme for data clustering}\label{sec:thescheme}

In this section, we provide the rigorous formulation of the MBO algorithm for data clustering originally given by Bertozzi et al. in~\cite{Merkurjev2013},~\cite{Gennip2014},~\cite{Merkurjev2014}. Let $G = (V, E)$ be a graph with vertex set $V = \{ x_1, . . . , x_n\}$ and let $E$ be the set of edges weighted by $w_{ij} = w_{ji},\ 1 \le i, j \le n$. We assume that $w_{ii} = 0$ for every $i = 1, . . . , n$. For every $i = 1, . . . , n$ we define the degree $d_i$ as

\begin{equation}
d_i := \frac{1}{n}\sum_{j =1}^n w_{ij}.\nonumber
\end{equation}
We will assume that $d_i > 0$ for every $1 \le i \le n$. We let $\mathcal{V}$ be the space of real valued functions defined on $V$. We define an inner product on $\mathcal{V}$ by

\begin{equation}
\langle u, v \rangle_\mathcal{V} = \frac{1}{n} \sum_{i = 1}^n d_i u_i v_i,\ u, v \in \mathcal{V}.\nonumber
\end{equation}
We let $\mathcal{E}$ be the space of antisymmetric functions on $E$. We define an inner product on $\mathcal{E}$ as
\begin{equation}
\langle F, G \rangle_{\mathcal{E}} = \frac{1}{2n^2}\sum_{i, j:\ w_{ij} \neq 0} F_{ij}G_{ij}\frac{1}{w_{ij}},\ F, G \in \mathcal{E}.\nonumber
\end{equation}
Given $\epsilon > 0$ we define the derivative operator $\nabla: \mathcal{V} \to \mathcal{E}$ acting on functions $u: V \to \mathbf{R}$ as

\begin{equation}\label{eq:derivative_op}
(\nabla u)_{ij} = \frac{w_{ij}}{\epsilon}(u_j - u_i).
\end{equation}
We denote by $\operatorname{div}: \mathcal{E} \to \mathcal{V}$ its adjoint with respect to the scalar products on $\mathcal{V}$ and $\mathcal{E}$. Explicitly, we may compute for $F \in \mathcal{E}$

\begin{equation}
(\operatorname{div}F)_i= \frac{1}{2\epsilon d_i} \sum_{j \neq i}\left( F_{ij} - F_{ji}\right).\nonumber
\end{equation} 
Finally, we introduce the \textit{random walk} graph Laplacian $\Delta := \operatorname{div} \circ \nabla: \mathcal{V} \to \mathcal{V}$. Explicitly, $\Delta$ can be identified with the matrix

\begin{equation}
\frac{1}{\epsilon^2}\left( \mathbb{I} - \frac{1}{n}D^{-1}W \right),\nonumber
\end{equation}
where $D = \operatorname{diag}(d_1, . . . , d_n)$ is the diagonal matrix of degrees and $W = (w_{ij})_{i,j=1}^n$ is the matrix of weights. Given $t \in \mathbf{R}$ we let $e^{-t\Delta}$ be the exponential of the matrix $-t\Delta$. If $u \in \mathcal{V}$ then the function $v(t) = e^{-t\Delta}u$ solves the heat equation with initial value $u$ on the graph, i.e.\

\begin{equation}
\begin{cases}
\frac{d}{dt} v(t)= -\Delta v(t),
\\
v(0) = u.
\end{cases}\nonumber
\end{equation}

We are now ready to introduce the MBO scheme for data clustering. Given a natural number $P \le n$, a classification of the points of $G$ into $P$ clusters is a function $\chi : V \to \{0, 1\}^P$ such that $\sum_{m=1}^P \chi^m(x) = 1$ for all $x \in V$. In other words, $\chi$ encodes a partition of the graph $G$ into $P$ clusters $C_m = \mathbf{1}_{\{\chi^m = 1\}}$, $m = 1, . . . , P$. Let $\sigma := (\sigma_{ml})_{1\le m,l \le P} \in \mathbf{R}^{P\times P}$ be a symmetric matrix with $\sigma_{ml} > 0$ for $m \neq l$ and $\sigma_{mm} = 0$. Then the MBO scheme for data clustering is as follows.

\begin{algorithm}[MBO scheme]\label{algo}
Let $P$ be the number of clusters, let $h > 0$ be the time-step size, and let $N$ be the number of iterations to run. Let $\chi_0:V \to \{ 0, 1\}^P$ be a given clustering of the graph into $P$ clusters. To obtain a clustering $\chi_N: V \to \{0, 1\}^P$ using the MBO scheme, define inductively a new clustering $\chi_{q+1}: V \to \{0, 1\}^P$ given the clustering $\chi_q: V \to \{0, 1\}^P$ by performing the following steps for $0 \le q < N$:
\begin{enumerate}
\item \textbf{Diffusion}. For every $m = 1, . . . , P$ define 
\begin{equation}
u_q^m := \sum_{m \neq l} \sigma_{ml} e^{-h\Delta}\chi_q^l.\nonumber
\end{equation}
\item \textbf{Thresholding}. For every $m = 1, . . . , P$ update the cluster by setting
\begin{equation}
\{ \chi_{q+1}^m = 1\} := \left\{ x \in V:\ u_q^m(x) < \min_{l \neq m} u_q^l(x)\right\}.\nonumber
\end{equation}
\end{enumerate}
\end{algorithm}

We define the set $\mathcal{M}_G := \left\{ u: V \to [0,1]^P:\ \sum_{m=1}^P u^m = 1\ \text{on}\ V \right\}$. For $u \in \mathcal{M}_G$ define the thresholding energy

\begin{equation}
E_h^G(u) = \frac{1}{\sqrt{h}} \sum_{i, j=1}^P \sigma_{ij} \langle u^i, e^{-h\Delta}u^j \rangle_\mathcal{V}.
\end{equation}
The following lemma, which is essentially due to Esedo\={g}lu and Otto, is the main motivation for our work. We also refer to~\cite[Proposition 4.6]{Gennip2014}.

\begin{lemma}[\cite{Esedoglu2015}]\label{lem:dissipationMBO}
Assume that the matrix $\sigma$ is negative semidefinite on $(1, . . . , 1)^{\perp}$, that means $v \cdot \sigma v \le 0$ for all $v \in \mathbf{R}^d$ with $v \cdot (1, . . . , 1) = 0$. In the setting of the previous algorithm, to obtain the new clustering $\chi_{q+1}: V \to \{0, 1\}^P$ starting from $\chi_q: V \to \{0, 1\}^P$ define
\begin{equation}
\chi_{q+1} \in \operatorname{argmin}_{u \in \mathcal{M}_G} \left\{ E_h^G(u) - \frac{1}{\sqrt{h}}\sum_{i, j=1}^{P} \sigma_{ij}\big\langle (u^i - \chi_q^i), e^{-h\Delta}(u^j - \chi_q^j) \big\rangle_{\mathcal{V}} \right\}.
\end{equation}
\end{lemma}

\section{Main results}\label{sec:mainres}

In this section we introduce the setting of our problem and we state our main results. For the technical background and definitions we refer to Section~\ref{sec:prelim}. 

We assume that $M \subset \mathbf{R}^d$ is a $k$-dimensional compact Riemannian submanifold of $\mathbf{R}^d$. Let $\nu = \rho\volm \in \mathcal{P}(M)$ be a probability measure on $M$, absolutely continuous with respect to the volume measure with a smooth and positive density $\rho$. Assume that $\{X_i\}_{i \in \mathbf{N}}$ are iid random points on $M$, distributed according to $\nu$. For any $n \in \mathbf{N}$ and $\epsilon > 0$ we define the random graph $G_{n, \epsilon}$ with vertex set given by $V_{n, \epsilon} = \left\{X_1, . . . , X_n\right\}$ and weights

\begin{equation}\label{eq:def_weight_rg}
w_{ij}^{(n, \epsilon)} = \frac{1}{\epsilon^k} \eta\left( \frac{| X_i - X_j |_d}{\epsilon}\right),\ i \neq j,
\end{equation}
and $w_{ii}^{(n, \epsilon)} = 0$, where $\eta: [0, +\infty) \to [0, +\infty)$ is a given function and $|\cdot |_d$ denotes the standard Euclidean norm on $\mathbf{R}^d$. Here $\epsilon$ is the same as in~\eqref{eq:derivative_op}. We stress that the graphs we constructed are random object and sometimes we will make this randomness explicit by specifying the dependence of the graph on an additional variable $\omega \in \Omega$, where $\Omega$ is the probability space on which the random objects $\{X_i\}_{i \in \mathbf{N}}$ are defined. On the function $\eta$ we set the following conditions:

\begin{enumerate}
\item $\eta(0) > 0$ and $\eta$ is continuous at $0$,
\item $\eta(t) \ge 0$ for every $t > 0$, $\eta$ is nonincreasing and $\eta$ is $C^2((0,+\infty))$,
\item $\eta, \eta', \eta''$ have exponential decay.
\end{enumerate}
Define $C_1 = \int_{\mathbf{R}^k} \eta(|y|_k)dy$ and $C_2 = \int_{\mathbf{R}^k} \eta(|y|_k)y_1^2dy$. Assume that $P \in \mathbf{N}$ and let $\sigma \in \mathbf{R}^{P \times P}$ be a symmetric matrix, negative definite on $(1, . . . , 1)^{\perp}$. We also assume that $\sigma_{ii} = 0$ for each $i = 1, . . . , P$ and that $\sigma_{ij} = \sigma_{ji} > 0$ for all $i \neq j$. Finally, we assume that the coefficients of $\sigma$ satisfy the \textit{triangle inequality}, that is

\begin{equation*}
\sigma_{ij} \le \sigma_{il} + \sigma_{lj}\quad \forall i,j,l \in \{1, . . . , P\}.
\end{equation*}
These assumptions are satisfied, for example, if we let $\sigma$ be the matrix defined by

\begin{equation*}
\sigma_{ij} = \begin{cases}
1\ &\text{if}\ i \neq j,
\\ 0\ &\text{otherwise}.
\end{cases}
\end{equation*}
The fact that the matrix $\sigma$ is negative definite is needed in order that the MBO scheme dissipates the thresholding energy at every iteration, cf.\ Lemma~\ref{lem:dissipationMBO}, while the triangle inequality ensures the lower semicontinuity of the energy. For each $n \in \mathbf{N}$ define the set

\begin{equation}
\mathcal{M}_n := \left\{ u: V_n \to [0,1]^P:\ \sum_{i=1}^P u^i = 1 \right\}.\nonumber
\end{equation}
Given $h > 0$, define the thresholding energies on $\mathcal{M}_n$ as

\begin{equation}\label{eq:thresh_energy_def}
E^h_{n, \epsilon}(u) := \frac{1}{\sqrt{h}} \sum_{i, j =1}^P \sigma_{ij} \langle u^i , e^{-h\Delta_{n, \epsilon}}u^j \rangle_{\mathcal{V}_n},\quad u \in \mathcal{M}_n.
\end{equation}
We also define the set

\begin{equation}
\mathcal{M} := \left\{ u: M \to [0,1]^P\ \text{measurable}:\ \sum_{i=1}^P u^i = 1\ \text{a.e.} \right\}.\nonumber
\end{equation}
Let $\mu = \xi\volm$ for $\xi \in C^{\infty}(M)$, $\xi > 0$. Given $h>0$ we define the thresholding energy on the weighted manifold $(M, \mu)$ as

\begin{equation}\label{eq:thresh_energy_manifold}
E_h(u) := \frac{1}{\sqrt{h}}\sum_{i,j=1}^P\sigma_{ij} \int_M u^i(x)e^{-h\Delta_{\xi}}u^j(x) d\mu,\quad u \in \mathcal{M}.
\end{equation}
Here, $\Delta_{\xi}$ is the weighted Laplacian on $(M, \mu)$, which is defined by its action on smooth functions $f \in C^{\infty}(M)$ as

\begin{equation}
\Delta_{\xi} f = -\frac{1}{\xi}\operatorname{div}\left( \xi \nabla f \right),
\end{equation}
and $e^{-t\Delta_{\xi}}$ is the corresponding heat operator; we refer to Section~\ref{sec:prelim} for the relevant background and definitions. We are now in a position to state our main results.

\begin{theorem}[Discrete to nonlocal $\Gamma$-convergence]\label{thm:discretenonloc}
Let $M$ be a $k$-dimensional compact Riemannian submanifold of $\mathbf{R}^d$. Let $\nu = \rho\volm \in \mathcal{P}(M)$ be a probability measure, absolutely continuous with respect to the volume element with a smooth and positive density. Let $\{X_i\}_{i \in \mathbf{N}}$ be a sequence of iid random points in $M$, distributed according to $\nu$. Let $\epsilon_n > 0$ be a sequence such that 
\begin{equation}\label{eq:assumptions_rates_epsilon}
\lim_{n \to +\infty} \frac{\epsilon_n^{k+2}n}{\operatorname{log}(n)} = +\infty,\ \lim_{n\to +\infty} \epsilon_n = 0.
\end{equation}
Let $G_{n, \epsilon_n}$ be the corresponding random graphs. It holds almost surely that for any $h > 0$ if $v_n$ converge weakly to $v$ in $TL^2(M)$ then
\begin{equation*}
\lim_{n\to +\infty} E_{n, \epsilon_n}^h(v_n) = \sqrt{\frac{C_1C_2}{2}}E_{\frac{C_2h}{2C_1}}(v),
\end{equation*}
where the thresholding energy on the right hand side corresponds to the weight $\xi = \rho^2$. Moreover, every sequence $v_n \in \mathcal{M}_n$ has a subsequence converging weakly in $TL^2(M)$ and every limit point lies in $\mathcal{M}$.
\end{theorem}

The main tool for proving Theorem~\ref{thm:discretenonloc} is the following strong convergence of the heat operators on the graphs to the corresponding heat operator on the manifold.

\begin{theorem}\label{thm:conv_heat}
Let the assumptions of Theorem~\ref{thm:discretenonloc} be in place. Then it holds almost surely that if $u_n \in \mathcal{V}_n$ is a sequence of functions converging weakly to $u \in L^2(M)$ in $TL^2)M)$, then for every $t > 0$ we have

\begin{equation}
\lim_{n \to +\infty} e^{-t\Delta_{n, \epsilon_n}}u_n = e^{-\frac{C_2}{2C_1} t \Delta_{\rho^2}}u\ \text{strongly in}\ TL^2(M).\nonumber
\end{equation}
\end{theorem}
As a Corollary of Theorem~\ref{thm:conv_heat} we also obtain the following result about the consistency of one step of MBO in the large data limit, which, in the setting of random geometric graphs, answers positively to a question by Bertozzi et al.~\cite[Question 7.5]{Gennip2014}.

\begin{corollary}\label{cor:bertozzi_qst}
Let the assumptions of Theorem~\ref{thm:discretenonloc} hold true. Then the following holds almost surely: Let $\chi_n: V_n \to \{0, 1\}^P$ be such that the sequence $\{\chi_n\}_{n \in \mathbf{N}}$ converges weakly in $TL^2(M)$ to a function $\chi: M \to \{0, 1\}^P$. Denote by $\chi_n^h: V_n \to \{0, 1\}^P$ the outcome of one step ($N=1$) of the MBO scheme (Algorithm~\ref{algo}) on the $n$-th graph with step size $h>0$ and initial clustering $\chi_n$. Denote by $\chi^h:M \to \{0, 1\}^P$ the outcome of one step of the MBO scheme on the manifold (Algorithm~\ref{algoman}) with initial value $\chi$, step size $h > 0$ and diffusion parameter $\kappa = \frac{C_2}{2C_1}$. Then the sequence $\{\chi_n^h\}_{n \in \mathbf{N}}$ converges weakly to $\chi^h$ in $TL^2(M)$.
By induction, the convergence holds for any number $N$ of iterations.
\end{corollary}

\begin{remark}
Let us point out that if $\Omega \subset M$ is an open set with smooth boundary, then the functions $\chi_n: V_n \to \{0, 1\}$ defined as $\chi_n := \mathbf{1}_{V_n \cap \Omega}$ are such that almost surely
\begin{equation}\label{eq:strong_tl}
TL^2(M)-\lim_{n \to +\infty} \chi_n = \chi := \mathbf{1}_{\Omega}.
\end{equation}
In particular, the conclusion of Corollary~\ref{cor:bertozzi_qst} holds true with these choices of initial values. Equation~\eqref{eq:strong_tl} follows by the fact that if $T_n$ is a sequence of transport maps obtained by applying Theorem~\ref{thm:transp_plan_existence}, then there exists $\delta > 0$ and a constant $C$ depending only on $M$ and $\rho$ such that if $\theta_n := \sup_{x \in M} d_M(x, T_n(x)) \le \delta$ then
\begin{equation}\label{eq:perimeter_bound}
\int_M |\chi_n(T_n(x)) - \chi(x)| d\nu \le C\theta_n \int_{\partial \Omega} \rho d\mathcal{H}^{k-1}.
\end{equation}
The validity of~\eqref{eq:perimeter_bound} is shown in the flat case in~\cite[Remark 5.1]{GarciaTrillos2016}, the analogous estimate on a closed manifold can be shown in a similar way.
\end{remark}

If $u \in \mathcal{M}$ is such that $u \in BV(M, \{0,1\}^P)$ then we set $\Omega_i := \{ u^i = 1\}$ and $\Sigma_{ij} := \partial^*\Omega_i \cap \partial^*\Omega_j$, the intersection of the reduced boundaries of $\Omega_i$ and $\Omega_j$. Again, we refer to Section~\ref{sec:prelim} for the relevant background. We have the following result about the convergence of the nonlocal thresholding energy on the manifold.

\begin{theorem}[Nonlocal to local $\Gamma$-convergence]\label{thm:nonlocloc}
Let $M$ be a $k$-dimensional compact Riemannian submanifold of $\mathbf{R}^d$ weighted by a measure $\mu = \xi \volm$ with $\xi > 0$, $\xi \in C^{\infty}(M)$. Let $\sigma \in \mathbf{R}^{P \times P}$ be symmetric, $\sigma_{ii} = 0$ and such that $\sigma$ satisfy the triangle inequality. Then on $\mathcal{M}$
\begin{equation}
\Gamma(L^1(M))-\lim_{h \downarrow 0} E_h = E,\nonumber
\end{equation}
where we define for $u \in \mathcal{M}$
\begin{equation}
E(u) = \begin{cases}
\frac{1}{\sqrt{\pi}}\sum_{ij} \sigma_{ij}|D\chi_{\Omega_i}|_{\xi}(\Sigma_{ij})\ &\text{if}\ u \in BV(M, \{0,1\}^P),
\\
+\infty\ &\text{otherwise}.
\end{cases}\nonumber
\end{equation}
Moreover, if $u \in BV(M, \{0,1\}^P) \cap \mathcal{M}$ then we have $\lim_{h \downarrow 0} E_h(u) = E(u)$. Finally, if $u_h $ are functions in $\mathcal{M}$ such that $\sup_{h>0} E_h(u_h) < +\infty$ then the family $\{ u_h \}_{h \downarrow 0}$ is precompact in $L^1(M)$ and every limit point is in $BV(M, \{0,1\}^P) \mathcal {M}$.
\end{theorem}

\begin{remark}\label{rem:geometric_ass}
It should be remarked that the geometric assumptions for the previous result are not sharp. For $x \in M$ and $r > 0$ we denote by $B_r(x)$ the Riemannian ball of radius $r$ centered at $x$. We also denote by $p$ the heat kernel for the weighted Laplacian $\Delta_\xi$ (cf.\ \eqref{eq:heat_semig}). For our proof to work, we need the following properties.
\begin{enumerate}[(i)]
\item \textbf{Doubling property}. There exists $N > 0$ such that for any $x \in M$ and any $r > 0$
\begin{equation}\label{eq:doubling}
\mu(B(x, 2r)) \le 2^N \mu(B(x,r)).
\end{equation}
\item \textbf{Asymptotic expansion for the heat kernel}. There exist functions $v_j \in C^{\infty}(M \times M), j \in \mathbf{N}$, such that for every $N> l + \frac{k}{2}$ there exists a constant $\tilde{C}_N < \infty$ such that
\begin{equation}\label{eq:asym_exp}
\left| \nabla^l \left( p(t, x, y) - \frac{e^{-\frac{d^2(x,y)}{4t}}}{(4\pi t)^{k/2}}\sum_{j=0}^{N}v_j(x,y)t^j\right) \right| \le \tilde{C}_N t^{N+1-\frac{k}{2}},
\end{equation} 
provided $d(x,y) \le \frac{\operatorname{inj}(M)}{2}$, where $\operatorname{inj}(M)$ is the injectivity radius of the manifold $M$. Moreover we have that $v_0(x,x) = \frac{1}{\xi(x)}$.
\item \textbf{Gaussian bounds I}. There exists constants $Q_1, Q_2, Q_3, Q_4 > 0$ such that for every $t > 0$ and all $x, y \in M$,
\begin{align}\label{eq:gauss_up_bds_I}
\frac{Q_1}{\mu(B_{\sqrt{t}}(x))} e^{-\frac{d^2(x,y)}{Q_2t}} \le p(t,x,y) \le \frac{Q_3}{\mu(B_{\sqrt{t}}(x))} e^{-\frac{d^2(x,y)}{Q_4t}}. 
\end{align}
\item \textbf{Gaussian bounds II}. There exist $\hat{C}_1,\hat{C}_2 > 0$ such that for any $x, y \in M$ and any $t > 0$
\begin{equation}\label{eq:gauss_up_bds_II}
|\nabla_x p(t,x,y)| \le \frac{\hat{C}_1}{\sqrt{t} \mu(B_{\sqrt{t}}(x))} \operatorname{exp}\left( -\frac{d^2(x,y)}{\hat{C}_2 t}\right).
\end{equation}
\end{enumerate}
These properties are satisfied in the case of a closed manifold as we work with. Indeed the doubling property follows from~\cite{Cheeger1982}, the asymptotic expansion holds by construction of the heat kernel via the parametrix method, cf.~\cite[Chapter 3]{Rosenberg1997}. Finally,~\eqref{eq:gauss_up_bds_I} and~\eqref{eq:gauss_up_bds_II} follow from the Li--Yau inequality for weighted manifolds~\cite{Setti1992}.
\end{remark}

\section{Semi-supervised learning}\label{sec:semisupervised}
Theorem~\ref{thm:discretenonloc} and Theorem~\ref{thm:nonlocloc} combined together prove the \textit{consistency} of the MBO scheme for data clustering: indeed the two $\Gamma$-convergence results prove that (local) minimizers of the energies $E_{n, \epsilon_n}^h$ converge to (local) minimizers of the weighted perimeter on the manifold if we let first $n \to +\infty$ and then $h \downarrow 0$. Of course, the only global minimizers for $E_{n, \epsilon_n}^h$ are partitions where all points are labeled in the same way, and thus the results may seem of little relevance. The full strength of Theorem~\ref{thm:discretenonloc} and Theorem~\ref{thm:nonlocloc} is seen in the context of \textit{semi-supervised} learning. In semi-supervised learning one is given:
\begin{enumerate}
\item A dataset of $n$ distinct points $D := \{x_1, . . ., x_n\} \subset \mathbf{R}^d$.
\item A number of classes $P \in \mathbf{N}$ to split the data into.
\item A subset $\mathcal{O} \subset D$ of $L \ll n$ points and a function $u_0: \mathcal{O} \to \{c_1, . . . , c_P\}$ which assigns a label to every point in $\mathcal{O}$.
\end{enumerate}
The task is to assign labels to all points in the dataset using both the known labels and the geometry of the dataset. The MBO scheme can be suitably modified to perform semi-supervised learning: The main point is to replace the heat operator in the first step of the algorithm by another differential operator with a fidelity term, see~\cite{Merkurjev2013} for the details. This yields an algorithm which still has a minimizing movements interpretation, but the associated energy involves a different operator than the heat semigroup for the Laplacian. A different approach looks at changing the thresholding value in the thresholding step of MBO while leaving the differential operator in the diffusion step unchanged. This modified version of MBO has the advantage that the energy in the variational formulation is just the "standard" thresholding energy plus a linear term. The ideas behind the SSL MBO algorithm come from the corresponding MBO scheme for forced mean curvature flow (see ~\cite{Mascarenhas} and~\cite{Laux2017}) and have already been adapted to data classification by Jacobs in~\cite{Jacobs2016}. We assume that we are given $D := \{x_1, . . ., x_n\} \subset \mathbf{R}^d$ data points to be classified into $P$ clusters. We assume that we are given a function $f: D \to \mathbf{R}^P$, the forcing term. As done in Section~\ref{sec:thescheme}, construct a similarity graph for the dataset. Then the MBO scheme for semi-supervised learning reads as follows.

\begin{algorithm}[SSL MBO]\label{algosemisup}
Let $h > 0$, let $N$ the number of iterations to run. Let $\chi_0 :V \to \{0,1\}^P$ be a proposed clustering. To obtain a clustering $\chi_N: V \to \{0,1\}^P$ using the MBO scheme define inductively for $0 \le q < N$ a new clustering $\chi_{q+1}:V \to \{0, 1\}^P$ starting from the clustering $\chi_q: V \to \{0, 1\}^P$ by performing the following two steps:

\begin{enumerate}
\item \textbf{Diffusion}. For every $i \in \{1, . . . , P\}$ define 
\begin{equation}
u^i = \sum_{j \neq i} \sigma_{ij}e^{-h\Delta}\chi_q^j
\end{equation}
\item \textbf{Thresholding}. Update, for every $1 \le i \le P$
\begin{equation}
\left\{ \chi_{q+1}^i = 1 \right\} := \left\{ u^i -\sqrt{h}f^i < u^j - \sqrt{h}f^j,\forall j \neq i\right\}.
\end{equation}
\end{enumerate}
\end{algorithm}

The reason behind this approach to SSL is that the previous algorithm has a variational interpretation which adds just a linear term to the thresholding energy, namely we have the following result.
\begin{lemma}
Each iteration of the SSL MBO scheme decreases the energy
\begin{equation}
F_h(u) = \frac{1}{\sqrt{h}}\sum_{i,j=1}^P \sigma_{ij}\langle u^i, e^{-h\Delta}u^j \rangle_{\mathcal{V}} - \sum_{i=1}^P \langle f^i, u^i \rangle_{\mathcal{V}}.
\end{equation}
\end{lemma}
It is then natural to investigate the asymptotic behavior of these energies in the sense of the following theorems.

\begin{theorem}\label{thm:discnonloclsemisup}
Under the assumptions of Theorem~\ref{thm:discretenonloc}, if we additionally assume that $f_n: G_n \to \mathbf{R}^P$ are such that $f_n \to f$ in $TL^2(M)$, then it holds almost surely that for every $h > 0$ it holds that if $v_n$ converge weakly to $v$ in $TL^2(M)$ then

\begin{equation}
\lim_{n \to +\infty} F_{n, \epsilon_n}^h(v_n) = \sqrt{\frac{C_1C_2}{2}}F_{\frac{C_2h}{2C_1}}(v),
\end{equation}
where we set 
\begin{align*}
&F_{n, \epsilon_n}^h(v) = \frac{1}{\sqrt{h}}\sum_{i,j=1}^P \sigma_{ij}\langle v^i, e^{-h\Delta_{n, \epsilon_n}}v^j\rangle_{\mathcal{V}_{n, \epsilon_n}} - \sum_{i=1}^P\langle f_n^i, v^i\rangle_{\mathcal{V}_{n, \epsilon_n}},
\\ &F_{h}(u) = E_{h}(u) -\sqrt{\frac{2C_1}{C_2}} \sum_{i=1}^P\int_M f^i u^i \rho^2d\volm\quad u \in \mathcal{M},
\end{align*}
where $\mathcal{M} := \{u:M \to [0,1]^P\ \text{measurable}\ :\sum_{i=1}^P u^i = 1\ \text{a.e.}\}$.
\end{theorem}
\begin{theorem}\label{thm:nonloclloclsemisup}
Let $M$ be a $k$-dimensional compact Riemannian submanifold of $\mathbf{R}^d$ weighted by a measure $\mu = \xi \volm$ with $\xi > 0$, $\xi \in C^{\infty}(M)$. Let $f \in L^1(M)$. Then on $\mathcal{M}$,

\begin{equation}
\Gamma(L^1(M))-\lim_{h \downarrow 0} F_h = F,
\end{equation}
where we define

\begin{equation*}
F(u) = \begin{cases}
\frac{1}{\sqrt{\pi}}\sum_{i,j=1}^P \sigma_{ij}|Du^i|_{\xi}(\Sigma_{ij}) - \sqrt{\frac{2C_1}{C_2}}\sum_{i=1}^P\int_M f^i u^i \rho^2d\volm &\text{if}\ u \in BV(M, \{0,1\})^P,
\\ +\infty &\text{otherwise}.
\end{cases}
\end{equation*}
\end{theorem}

Theorem~\ref{thm:discnonloclsemisup} and Theorem~\ref{thm:nonloclloclsemisup} are easy consequences of Theorem~\ref{thm:discretenonloc}, Theorem~\ref{thm:nonlocloc} and the stability of $\Gamma$-convergence with respect to continuous perturbations. Of course, these Theorems prove the consistency of SSL MBO once one can produce suitable forcing functions $f_n$ which have a limit in $TL^2(M)$. In the following, let us for simplicity focus on the simple two-class setting, in which the cluster is $C = \left\{ u > \frac{1}{2} - \sqrt{h}f \right\}$. For a fixed $n \in \mathbf{N}$, if one is given a labeling function $u_0: \mathcal{O} \to \{0,1\}$, a very intuitive choice of forcing would be

\begin{equation}\label{eq:formula_forcings}
f = -\gamma(1-2u_0)\mathbf{1}_{\mathcal{O}},
\end{equation}
for some fixed constant $\gamma>0$. Indeed, when $x \in \mathcal{O}$ and $u_0(x) = 0$, $f(x_0) = -\gamma$ so that the thresholding value for $x$ gets higher and thus the updated set is more likely not to contain $x$. Observe in particular that if one chooses  $\gamma > \frac{1}{\sqrt{h}}$ then the values on $\mathcal{O}$ are enforced. Similarly the case $u_0(x) = 1$ forces the updated sets to contain $x$. The problem with such a strategy is that this is ill-posed in the large data limit: indeed, to talk about SSL one usually assumes that

\begin{equation}
\lim_{n \to +\infty} \frac{|\mathcal{O}_n|}{n} = 0,
\end{equation}
i.e., that the proportion of labeled points converges to zero. If we were now given labels $u_0^{(n)}: \mathcal{O}_n \to \{0,1\}$ then the corresponding forcings constructed according to formula~\eqref{eq:formula_forcings} would converge to zero in $TL^1(M)$, thus the large data limit forgets the labels. To overcome this difficulty one has first to propagate the labels to the whole graph to produce a new forcing function. There are several strategies for doing this: for example, Jacobs uses a forcing fidelity term based on the graph geodesic distance, see~\cite{Jacobs2016}. Here we propose to construct the forcing function by means of Lipschitz learning.

\begin{algorithm}[SSL MBO with Lipschitz learning]\label{algosemisup_ll}
Let $h > 0$, let $N$ be the number of iterations to run. Let $C$ be a proposed clustering. Let $u_0: \mathcal{O} \to \{0,1\}$ given labels for a subset $\mathcal{O} \subset D$ of data points.
\begin{enumerate}
\item \textbf{Lipschitz learning - forcing construction}. Use Lipschitz learning (eventually using a reweighted graph with self tuning weights) to propagate the given labels, i.e.\ find $u: D \to \mathbf{R}$ such that

\begin{equation}
\begin{cases}
\Delta^{(\infty)}u = 0\ &\text{on}\ D\setminus \mathcal{O},
\\ u = u_0\ &\text{on}\ \mathcal{O}.
\end{cases}
\end{equation}
Then set $f = -\gamma(1-2u)$ for some given constant $\gamma>0$. Here $\Delta^{(\infty)}$ denotes the infinity Laplacian on the graph.
\item \textbf{SSL MBO}. Perform $N$ iterations of the SSL MBO Algorithm~\ref{algosemisup} with initial clustering $C$ and forcing $f$.
\end{enumerate}
\end{algorithm}
The reason why Lipschitz learning is a good approach to generate the forcing is because it is a very well-posed algorithm in the large data limit. Indeed let us for simplicity consider the case when $M = \mathbb{T}^k$, the $k$-dimensional torus. Fix a set $\mathcal{O} \subset \{ X_i \}_{i \in \mathbf{N}}$ and assume that one is given a labeling function $u_0: \mathcal{O} \to \{0,1\}$. Denote by $\Delta^{(\infty)}_{n}$ the infinity Laplacian on the $n$-th graph, i.e.\ the operator which acts on $u: V_n \to \mathbf{R}$ as

\begin{equation}
\Delta_n^{(\infty)}u(i) = \max_{1 \le j \le n} w_{ij}^{(n, \epsilon_n)}(u(x_j) - u(x_i)) + \min_{1 \le j \le n} w_{ij}^{(n, \epsilon_n)}(u(x_j) - u(x_i)).
\end{equation}
Define $u_n: V_n \to \mathbf{R}$ as solutions of

\begin{equation}
\begin{cases}
\Delta_{n}^{(\infty)}u_n = 0\ &\text{on}\ V_n\setminus \mathcal{O},
\\ u_n = g\ &\text{on}\ \mathcal{O}.
\end{cases}
\end{equation}
Calder showed in~\cite{Calder2019} that $u_n$ converges uniformly to $u \in C^{0,1}(M)$, the unique viscosity solution of the $\infty$-Laplace equation

\begin{equation}
\begin{cases}
\Delta^{(\infty)}u := \sum_{l,m = 1}^k \partial_l u\partial_{lm}u \partial_m u= 0\ &\text{on}\ M\setminus \mathcal{O},
\\ u = g\ &\text{on}\ \mathcal{O}.
\end{cases}
\end{equation}
In particular, $f_n := -\gamma(1-2u_n)$ converges to $f := -\gamma(1-2u)$ in $TL^2(M)$ and the assumptions of Theorem~\ref{thm:discnonloclsemisup} are satisfied. 

\section{Discussion}\label{sec:discussions}
\subsection{Joint limit and monotonicity}

We want to remark that it would be interesting to understand whether we can take the joint limit $n \to +\infty$ and $h \downarrow 0$, combing Theorem~\ref{thm:discretenonloc} and Theorem~\ref{thm:nonlocloc} to give that

\begin{equation}\label{eq:joint_lim}
\Gamma-\lim_{n \to +\infty} E^{h_n}_{n, \epsilon_n} = E,
\end{equation}
where the thresholding energies $E^{h_n}_{n, \epsilon_n}$ are defined in~\eqref{eq:thresh_energy_def} and $E$ is defined in the statement of Theorem~\ref{thm:nonlocloc}. Here the $\Gamma$-limit has to be understood in the sense of $TL^1(M)$ convergence. At the present moment, we are not able to prove~\eqref{eq:joint_lim}. However, let us sketch a possible approach for obtaining the $\Gamma$-$\liminf$ inequality for~\eqref{eq:joint_lim}. Assume that we knew that for fixed $n \in \mathbf{N}$ and for $\tilde{h} \ge h$

\begin{equation}\label{eq:monotonicity_easy}
E_{n, \epsilon_n}^{\tilde{h}}(u) \le E_{n, \epsilon_n}^{h}(u),\quad u \in \mathcal{M}_n.
\end{equation}
Then the $\Gamma$-$\liminf$ inequality in~\eqref{eq:joint_lim} would follow from Theorem~\ref{thm:discretenonloc} and Theorem~\ref{thm:nonlocloc}. Indeed, assume that $u_n \in \mathcal{M}_n$ are such that $u_n \to u \in \mathcal{M}$ in $TL^1(M)$. Fix $m \in \mathbf{N}$. Since $h_n \downarrow 0$, we have that $h_n \ll h_m$ for $n$ large enough. Thus by~\eqref{eq:monotonicity_easy} and by Theorem~\ref{thm:discretenonloc} we would get

\begin{equation*}
E^{h_m}(u) = \liminf_{n \to +\infty} E_{n, \epsilon_n}^{h_m}(u_n) \le \liminf_{n \to +\infty} E_{n, \epsilon_n}^{h_n}(u_n).
\end{equation*}
Now letting $m \to +\infty$ and using the consistency Theorem~\ref{thm:nonlocloc} one would get

\begin{equation*}
E(u) \le \liminf_{n \to +\infty} E_{n, \epsilon_n}^{h_n}(u_n).
\end{equation*}
This means that a key ingredient for the joint limit is the monotonicity~\eqref{eq:monotonicity_easy}. Of course, also an approximate version of it would suffice. For example

\begin{equation}\label{eq:monotonicity_relaxed}
E_{n,\epsilon_n}^{\tilde{h}}(u) \le g(h)E_{n, \epsilon_n}^{h}(u) + f(\tilde{h})E_{n, \epsilon_n}^{h}(u) + z(\tilde{h})
\end{equation}
where $g, f, z$ are functions such that $\lim_{h \downarrow 0} g(h) = 1$, $\lim_{\tilde{h} \downarrow 0} f(\tilde{h}) = 0$ and $\lim_{\tilde{h} \downarrow 0} z(\tilde{h}) = 0$. The reason behind the hope for a monotonicity property for the discrete thresholding energies comes from the similar property which holds in the continuum in the Euclidean setting, see Lemma A.2 in~\cite{Esedoglu2015}. Actually, an approximate version of this monotonicity is true also for the localized thresholding energies, see Theorem~\ref{thm:monotonicitylocalized} in the \hyperlink{sec:appendix}{Appendix}. By exploiting this result, using a suitable localization argument and the asymptotic expansion for the heat kernel~\eqref{eq:asym_exp} one can actually prove a similar monotonicity property for the thresholding energies on the manifold. Since the discrete thresholding energies are approximating the thresholding energy on the manifold, it is reasonable to believe that such a property holds true in some sense also at the discrete level. 
To support this idea, we run some numerical experiments. Quite surprisingly, it seems that the validity of this property is related to the rate $\frac{h}{\epsilon_n^2}$, in particular, it does not hold when $h \ll \epsilon_n^2$ and it seems to hold for $h \gg \epsilon_n^2$. Observe that the regime $h \gg \epsilon_n^2$ is the one which is relevant for applications, because for $h \ll \epsilon_n^2$ the MBO scheme is pinned, see~\cite{Gennip2014}. It is not too difficult to show that the energies $E_{n, \epsilon_n}^h$ are actually \textit{increasing} if $h \ll \epsilon_n^2$. A simple numerical experiment that we run is as follows: sample $n$ data points from the uniform distribution on the unit sphere, see Figure~\ref{fig:sample_sphere}. Construct the similarity graph with weight functions $\eta(t) = e^{-t^2}$, randomly choose a $\{0,1\}$-valued function $u$ which takes the value $1$ on half of the data  points, and then compute the thresholding energies $E_{n, \epsilon_n}^{h}(u)$ for $h \in \{2^{-5} \epsilon_n^2, . . . , 2^{4}\epsilon_n^2 \}$. The results are depicted in Figure~\ref{fig:thr_en}. We see that when $h \gg \epsilon^2$ the monotonicity seems to hold true: we experimented the same behavior also when using different distributions for the data points and different choices of functions $u$.
\begin{figure}[H]
\centering
\begin{minipage}{.5\textwidth}
  \centering
\includegraphics[width=.9\linewidth]{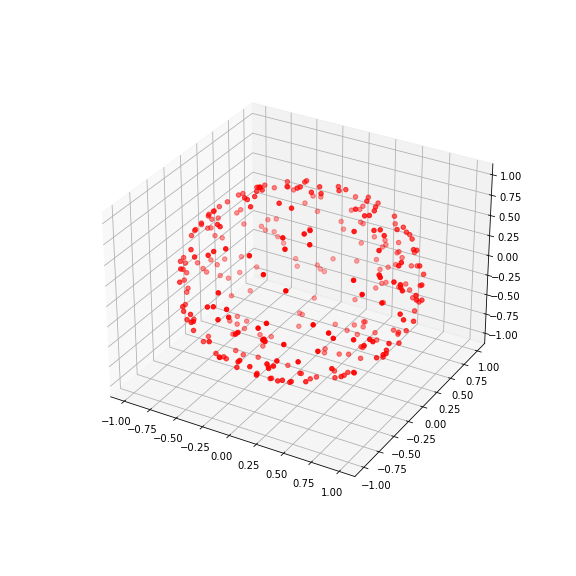}
  \caption{Sample points from uniform distribution on the unit sphere.}
  \label{fig:sample_sphere}
\end{minipage}%
\begin{minipage}{.5\textwidth}
  \centering
\includegraphics[width=.9\linewidth]{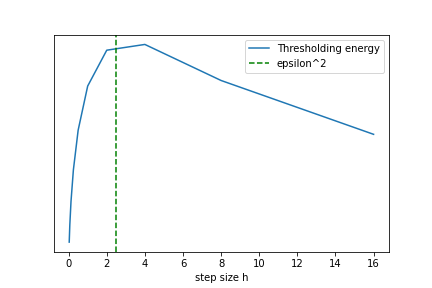}
  \caption{The thresholding energy $E_{n, \epsilon_n}^h$ for different values of $h$.}
  \label{fig:thr_en}
\end{minipage}
\end{figure}

\subsection{Extensions}\label{sec:extensions}

Here, we summarize the necessary changes to extend our results to other choices of graph Laplacians and to data dependent weights. With the notation and in the setting of Section~\ref{sec:mainres}, given $\lambda \in \mathbf{R}$ we define

\begin{equation}\label{eq:datadepweights}
w_{ij}^{(n, \epsilon, \lambda)} = \frac{w_{ij}^{(n, \epsilon)}}{\bigg( d_i^{(n, \epsilon)} d_j^{(n, \epsilon)}\bigg)^\lambda},
\end{equation}
and
\begin{equation}\label{eq:degreesdepweights}
d_i^{(n, \epsilon, \lambda)} = \frac{1}{n} \sum_{j=1}^n w_{ij}^{(n, \epsilon, \lambda)}.
\end{equation}

We can then consider the weighted graph $G_{n, \epsilon, \lambda}$ where the vertex set is given by $V_{n} := \{ X_1, . . . , X_n \}$ and the weights are given by~\eqref{eq:datadepweights}. Let $W^{(n, \epsilon, \lambda)}$ be the matrix of weights, and $D^{(n, \epsilon, \lambda)}$ the diagonal matrix of the degrees~\eqref{eq:degreesdepweights}. We then consider the following operators on $\mathcal{V}_{n, \epsilon, \lambda}$:

\begin{align*}
\Delta_{n, \epsilon, \lambda}^{rw} u &= \frac{1}{\epsilon^2}\left( \mathbb{I} - \frac{1}{n}(D^{(n, \epsilon, \lambda)})^{-1} W^{(n, \epsilon, \lambda)} \right)u,\ &u \in \mathcal{V}_{(n, \epsilon, \lambda)},
\\ \Delta_{n, \epsilon, \lambda}^{un} u &= \frac{1}{\epsilon^2}\left( D^{(n, \epsilon, \lambda)} - \frac{1}{n}W^{(n, \epsilon, \lambda)} \right)u,\ &u \in \mathcal{V}_{(n, \epsilon, \lambda)}.
\end{align*}
We define inner products on $\mathcal{V}_{(n, \epsilon, \lambda)}$ as

\begin{align*}
\langle u, v \rangle_{\mathcal{V}_{(n, \epsilon, \lambda)}, rw} &= \frac{1}{n} \sum_{i=1}^n d_{i}^{(n, \epsilon, \lambda)} u_i v_i,\ &u, v \in \mathcal{V}_{(n, \epsilon, \lambda)},
\\ \langle u, v \rangle_{\mathcal{V}_{(n, \epsilon, \lambda)}, un} &= \frac{1}{n} \sum_{i=1}^n u_i v_i,\ &u, v \in \mathcal{V}_{(n, \epsilon, \lambda)}.
\end{align*}
We also define inner products on $\mathcal{E}_{(n, \epsilon, \lambda)}$ as

\begin{align*}
\langle F, G \rangle_{\mathcal{E}_{(n, \epsilon, \lambda)}} &= \frac{1}{2n^2} \sum_{i, j:\ w_{ij}^{(n, \epsilon, \lambda)}\neq 0} F_{ij} G_{ij} \frac{1}{w^{(n, \epsilon, \lambda)}_{ij}},\quad F, G \in \mathcal{E}_{(n, \epsilon, \lambda)}.
\end{align*}
We define the Dirichlet energies

\begin{equation}\label{eq:dirichletdatadep}
E_{(n, \epsilon, \lambda)}(u) = \frac{1}{2} |\nabla u|_{\mathcal{E}_{(n, \epsilon, \lambda)}}^2,\quad u \in \mathcal{V}_{(n, \epsilon, \lambda)}.
\end{equation}
Finally, we define the thresholding energies

\begin{align*}
E_{(n, \epsilon, \lambda, rw)}^h(u) &= \frac{1}{\sqrt{h}} \sum_{i,j=1}^P \sigma_{ij} \langle u^i, e^{-h\Delta_{(n, \epsilon, \lambda)}^{rw}}u^j \rangle_{\mathcal{V}_{(n, \epsilon, \lambda, rw)}}, \quad u \in \mathcal{M}_n,
\\ E_{(n, \epsilon, \lambda, un)}^h(u) &= \frac{1}{\sqrt{h}} \sum_{i,j=1}^P \sigma_{ij} \langle u^i, e^{-h\Delta_{(n, \epsilon, \lambda)}^{un}}u^j \rangle_{\mathcal{V}_{(n, \epsilon, \lambda, un)}}, \quad u \in \mathcal{M}_n.
\end{align*}
We then have the following results.

\begin{theorem}\label{thm:convheatdatatep}
Let $\lambda \in \mathbf{R}$. Under the assumptions of Theorem~\ref{thm:discretenonloc} it holds almost surely that if $u_n \in \mathcal{V}_{(n, \epsilon, \lambda)}$ is a sequence of functions converging weakly to $u \in L^2(M)$ in $TL^2$, then for every $t>0$ we have
\begin{align*}
\lim_{n\to +\infty} e^{-t\Delta^{rw}_{(n, \epsilon_n, \lambda)}}u_n &= e^{-\frac{C_2}{2C_1} t \Delta_{\rho^s}}u\ \text{strongly in}\ TL^2,
\\ \lim_{n\to +\infty} e^{-t\Delta^{un}_{(n, \epsilon_n, \lambda)}}u_n &= e^{-\frac{C_2}{2C_1^{2\lambda}}t\rho^{1-2\lambda}\Delta_{\rho^s}}u\ \text{strongly in}\ TL^2,
\end{align*}
where $s = 2(1-\lambda)$.
\end{theorem}

\begin{theorem}\label{thm:discnonlocdatadep}
Under the assumptions of Theorem~\ref{thm:discretenonloc}, for every $\lambda \in \mathbf{R}$ it holds almost surely that for each fixed $h>0$, if $v_n$ converges weakly to $v$ in $TL^2(M)$,
\begin{align*}
\lim_{n \to +\infty} E_{(n, \epsilon_n, \lambda, rw)}^h(v_n) &= C_1^{\frac{1}{2}-2\lambda}\sqrt{\frac{C_2}{2}}E^{rw}_{\frac{C_2h}{2C_1}, \lambda}(v),
\\ \lim_{n \to +\infty} E_{(n, \epsilon_n, \lambda, un)}^h(v_n) &= \sqrt{\frac{C_2}{2C_1^{2\lambda}}}E^{un}_{\frac{C_2}{2C_1^{2\lambda}}, \lambda}(v),
\end{align*}
where we define, for $v \in \mathcal{M}$,

\begin{align*}
E^{rw}_{h, \lambda}(v) &:= \frac{1}{\sqrt{h}} \sum_{i,j=1}^P \sigma_{ij} \int_M u^i e^{-h\Delta_{\rho^s}}u^j \rho^s d\volm,
\\ E^{un}_{h, \lambda}(v) &:= \frac{1}{\sqrt{h}} \sum_{i,j=1}^P \sigma_{ij} \int_M u^i e^{-h\rho^{1-2\lambda}\Delta_{\rho^s}}u^j \rho d\volm,
\end{align*}
with $s = 2(1-\lambda)$.
\end{theorem}

The proofs of Theorem~\ref{thm:convheatdatatep} and Theorem~\ref{thm:discnonlocdatadep} are completely analogous to the proofs of Theorem~\ref{thm:conv_heat} and Theorem~\ref{thm:discretenonloc} (which consider the random walk Laplacian with $\lambda = 0$). The only needed changes are:
\begin{enumerate}
\item Replace the use of Theorem~\ref{thm:conv_lap} with the convergence of the corresponding Laplacian (see~\cite[Theorem 30]{Hein2007}).
\item Replace the use of Theorem~\ref{thm:gamma_dirichlet} with the analogous statement for the Dirichlet energies~\eqref{eq:dirichletdatadep}. It seems that this is not written down anywhere in the literature, but the proof of Garc\'{i}a Trillos and Slep\v{c}ev~\cite{GarciaTrillos2018} should be easily adapted to this setting.
\end{enumerate}
Using Theorem~\ref{thm:discnonlocdatadep} one can clearly extend also the analogous statement for the semi-supervised MBO scheme.

\section{Preliminaries}\label{sec:prelim}

\subsection{Weighted manifolds}
Hereafter, $M = (M, g, \mu)$ will be a compact Riemannian manifold with $\partial M = \emptyset$. We will assume that $\mu = \xi \volm$ for some $\xi \in C^{\infty}(M)$ such that $\xi >0$. 

For every $x \in M$, we denote by $\langle \cdot, \cdot \rangle_x$ the inner product on $T_xM$ induced by the metric $g$, i.e., for any $v, w \in T_xM$ we have $\langle v, w \rangle_x = g_x(v, w)$. Let $f: M \to \mathbf{R}$ be a smooth function. Then the gradient $\nabla f(x) \in T_xM$ is defined uniquely by the relation

\begin{equation}
\langle \nabla f(x), Y \rangle_x = d_xf(Y)\ \forall Y \in T_xM.\nonumber
\end{equation} 
Let $\Gamma(TM)$ be the space of smooth vector fields on $M$. We can define the (weighted) divergence operator $\divv_{\xi}: \Gamma(TM) \to C^{\infty}(M)$ by the requirement that for any $f \in C^{\infty}(M)$ and $Y \in \Gamma(TM)$

\begin{equation}
\int_M \langle \nabla f(x), Y(x) \rangle_x d\mu(x) = -\int_M f(x) \divv_{\xi}Y(x) d\mu(x).\nonumber
\end{equation}
It is easy to check that the divergence can be expressed in local coordinates as

\begin{equation}
\divv_{\xi} Y = \frac{1}{\xi \sqrt{\operatorname{det}(g)}}\sum_{i=1}^k\partial_i \left( \xi \sqrt{\operatorname{det}(g)}Y^i\right).\nonumber
\end{equation}
We also define the weighted Laplacian $\Delta_{\xi}: C^{\infty}(M) \to C^{\infty}(M)$ as $\Delta_{\xi} = -\divv_{\xi} \circ \nabla$. A distribution on $M$ is a continuous linear functional $T: C^{\infty}(M) \to \mathbf{R}$. We denote by $\mathcal{D}'(M)$ the space of distributions on $M$. We follow the terminology of~\cite{Grigoryan2009} and say that a distributional vector field is a continuous linear functional $V: \Gamma(TM) \to \mathbf{R}$. If $V$ is a distributional vector field, we define its divergence as the distribution $\divv_{\xi}V \in \mathcal{D}'(M)$ such that $\divv_{\xi}V(f) = -\langle V, \nabla f \rangle$. We define the Sobolev space

\begin{equation}
W^{1,2}(M) := \left\{ u \in L^2(M, \mu):\ \nabla u \in L^2(TM, \mu) \right\},\nonumber
\end{equation}
which is a Hilbert space when endowed with the inner product

\begin{equation}
(u, v)_{W^{1,2}(M)} = (u, v)_{L^2(M, \mu)} + (\nabla u, \nabla v)_{L^2(M, \mu)}.\nonumber
\end{equation}
We also denote $W^{1,2}(M)$ by $H^1(M)$. We denote by $H^{-1}(M)$ its dual. If $T$ is a distribution, we define $\Delta_{\xi} T \in \mathcal{D}'(M)$ by requiring $\Delta_{\xi} T(f) = T(\Delta_{\xi}(f))$. In particular if $u \in L^2(M, \mu)$, then $\Delta_{\xi}u \in \mathcal{D}'(M)$. Define

\begin{equation}
\mathcal{W}^{2,2}(M) = \left\{ u \in W^{1,2}(M):\ \Delta_{\xi}u \in L^2(M, \mu)\right\}.\nonumber
\end{equation}
It is a standard result that $\Delta_{\xi}$ can be extended uniquely to a self-adjoint operator on $\mathcal{W}^{2,2}(M)$, see for instance~\cite[Theorem 4.6]{Grigoryan2009}. It can be shown that $\Delta_{\xi}$ is a nonnegative self-adjoint operator in $L^2(M)$ and $\operatorname{spec}(\Delta_{\xi}) \subset [0, +\infty)$. For $u \in L^2(M, \mu)$ we denote by $T(t)u = v(t, \cdot)$ the solution to the Cauchy problem

\begin{equation}
\begin{cases}
\partial_t v = -\Delta_{\xi}v\ &\text{in}\ (0, +\infty) \times M,
\\
v(0, x) = u(x) &\text{on}\ M.
\end{cases}\nonumber
\end{equation}
More precisely, the map $t \in (0, +\infty) \mapsto T(t)u = v(t, \cdot) \in L^2(M)$ is characterized by the following properties:

\begin{itemize}
\item It is strongly differentiable in $L^2(M)$.
\item For every $t > 0$ we have $T(t)u \in \operatorname{dom}(\Delta_{\xi})$ and
\begin{equation}
\frac{dT(t)u}{dt} = -\Delta_{\xi}T(t)u.\nonumber
\end{equation} 
\item $T(t)u \to u$ in $L^2(M)$ as $t \downarrow 0$.
\end{itemize}
One way of constructing $T(t)$ is by means of the spectral resolution of $\Delta_{\xi}$. I.e., one defines linear operators $T(t): L^2(M) \to L^2(M)$ by 

\begin{equation}
T(t) := \int_{0}^\infty e^{-t\gamma} dE_{\gamma},\nonumber
\end{equation}
where $E_{\gamma}$ is the spectral resolution of $\Delta_{\xi}$. We refer to~\cite[Chapter 7]{Grigoryan2009} for the details. Furthermore, one can show that there exists a smooth map $p: (0, +\infty) \times M \times M \to \mathbf{R}$ such that for any $u \in L^2(M)$ and every $t > 0$

\begin{equation}\label{eq:heat_semig}
e^{-t\Delta_{\xi}}u(x) := T(t)u(x) = \int_M p(t, x, y)u(y)d\mu(y).\nonumber
\end{equation}
We call $p$ the heat kernel for $\Delta_{\xi}$.

Another more constructive way to prove the existence of the heat kernel is by the so-called \textit{parametrix method}. This has the advantage of giving immediately the asymptotic expansion~\eqref{eq:asym_exp}. However, the construction is technical and we think it is not worth sketching it here. The reader is referred to~\cite[Chapter 3]{Rosenberg1997}, where this construction is carried out in detail for the case of constant density $\xi = 1$.

\subsection{The MBO scheme on weighted manifolds}\label{subsec:mboschememan}

In this subsection, we recall the MBO scheme on weighted manifolds, which can be used to approximate the evolution by multiphase (weighted) mean curvature flow. Hereafter $M$ is a $k$-dimensional closed Riemannian manifold endowed with a weight $\xi \in C^{\infty}(M)$, $\xi > 0$.

\begin{algorithm}[MBO scheme on manifolds]\label{algoman}
Let $P$ be the number of phases, let $h > 0$ be a time-step size and let $\kappa > 0$ be a diffusion coefficient. Let $\chi_0:M \to \{0,1\}^P$ be a partition of $M$ into $P$ phases. To obtain an approximation of the evolution of $\chi_0$ by multiphase mean curvature flow define inductively a new partition $\chi_{n+1}:M \to \{0, 1\}^P$ starting from $\chi_n:M \to \{0, 1\}^P$ by performing the following steps:
\begin{enumerate}
\item \textbf{Diffusion}. For every $m = 1, . . . , P$ define 
\begin{equation}
u_m^n := \sum_{m \neq l} \sigma_{ml} e^{-\kappa h\Delta_{\xi}}\chi_n^l.\nonumber
\end{equation}
\item \textbf{Thresholding}. Define a new partition $\chi_{n+1}:V \to \{0, 1\}^P$ by defining, for every $m = 1, . . . , P$
\begin{equation}
\left\{\chi_{n+1}^m = 1 \right\} := \left\{ x \in M:\ u_n^m(x) < \min_{l \neq m} u_n^l(x)\right\}.\nonumber
\end{equation}
\end{enumerate}
\end{algorithm}

We then have the following minimizing movements interpretation for the previous algorithm.

\begin{lemma}[\cite{Esedoglu2015}]\label{lem:minmovman}
Assume that $\sigma$ is negative semidefinite on $(1, . . . , 1)^{\perp}$, which means that $v \cdot \sigma v \le 0$ for all $v \in \mathbf{R}^d$ with $v \cdot (1, . . . , 1)^{\perp} = 0$. Given a step-size $h > 0$ and a diffusion parameter $\kappa > 0$, to obtain the new partition $\chi_{n+1}:M \to \{0, 1\}^P$ starting from $\chi_n: M \to \{0, 1\}^P$ one can define 
\begin{equation}
\chi_{n+1} \in \operatorname{argmin}_{u \in \mathcal{M}} \left\{ \sqrt{\kappa} E_{\kappa h}(u) - \frac{1}{\sqrt{h}}\sum_{i \neq j}\sigma_{ij}\int_M (u_i - \chi^{n}_i) e^{-\kappa h\Delta_\xi}(u_j - \chi_j^n) \xi d\volm\right\}.
\end{equation}
\end{lemma}

\subsection{BV functions on weighted manifolds}

We begin this section introducing the total variation $|Du|_{\xi}$ of a function $u \in L^1(M)$. We define

\begin{equation}
|Du|_{\xi}(M) = \sup\left\{ \int_{M} u \divv_{\xi} Y d\mu:\ Y \in \Gamma(TM), |Y| \le 1 \right\}.\nonumber
\end{equation}
We say that $u \in L^1(M)$ is in $BV(M)$ provided $|Du|_\xi(M) < +\infty$. One can prove the following result.

\begin{theorem}\label{thm:finite_per_man}
Let $u \in BV(M)$, then there exist a Radon measure $|Du|_{\xi} \in \mathcal{M}_+(M)$ and a $|Du|_{\xi}$-measurable vector field $\sigma_u$ such that $|\sigma_u| = 1\ |Du|_{\xi}$-almost everywhere and such that

\begin{equation}\label{eq:ibp}
\int_M u \divv_\xi X d\mu = - \int_M \langle \sigma_u, X \rangle d|Du|_{\xi}\ \forall X \in \Gamma(TM).
\end{equation}
\end{theorem}
The proof of the theorem is an adaptation of the classical Riesz representation theorem: first one works locally on an open set $V \subset M$ using an orthonormal frame $\{E_1, . . . , E_k\}$. Following the same lines of the proof of the Riesz representation theorem one can check that there exist a Radon measure $\gamma_V \in \mathcal{M}_+(V)$ and a $\gamma_V$ measurable vector field $\sigma^V_u$ such that $|\sigma_u^V| = 1\ \gamma_V$-a.e.\ and such that~\eqref{eq:ibp} holds true for all $X \in \Gamma(TV)$. Then one checks that if $V_1, V_2$ are two open subsets of $M$, the construction is consistent on $V_1 \cap V_2$. One can then take a covering $\{V_i\}_{i}$ and apply the construction on each element of the covering. Taking a partition of unity $\{\rho_i\}$ subordinate to the covering one defines

\begin{itemize}
\item $|Du|_{\xi}(W) := \sum_i \int_{W \cap V_i} \rho_i d\gamma_{V_i}$.
\item $\sigma|_{V_i} := \sigma_{V_i}$.
\end{itemize}
A subset $E \subset M$ is said to be of \textit{finite perimeter} if $\chi_E \in BV(M)$. If $E$ is a set of finite perimeter, we denote by $\operatorname{Per}_{\xi}(E) := |D\chi_E|_{\xi}(M)$ its perimeter. We will make use of the following elementary lemma, which follows easily from Theorem~\ref{thm:finite_per_man}.

\begin{lemma}\label{lemma:chart_fp}
Let $u \in L^1(M)$. Then $u \in BV(M)$ if and only if for every chart $(V, \psi)$ the map $u \circ \psi^{-1}$ is in $BV(\psi(V))$. In that case we have that for any chart $(V, \psi)$
\begin{equation}\label{eq:push_forw_bv}
\psi_{\#}|Du|_{\xi} = \gamma |D(u\circ \psi^{-1})|,
\end{equation}
where $\gamma = \xi \circ \psi^{-1} \sqrt{\operatorname{det}g^{ij}}$.
\end{lemma}

\begin{remark}\label{rem:red_bdry}
For a set of finite perimeter $E \subset M$, we define its reduced boundary $\partial_M^* E$ as follows:

\begin{equation}\label{eq:red_bdry}
\partial_M^* E = \left\{ x \in M: \exists (V, \psi)\ \text{chart of}\ M\ \text{s.t.}\ \psi(x) \in \partial^*\psi(E)\right\},
\end{equation}
where $\partial^*\psi(E)$ is the reduced boundary of the set $\psi(E)$ in the usual Euclidean setting. From now on, we will also denote by $\partial^* E$ the set $\partial^*_M E$. One can check that:
\begin{itemize}
\item The definition is well posed.
\item $|D\chi_E|_{\xi}$ is concentrated on $\partial^* E$. In particular $|D\chi_E|_{\xi}$-a.e.\ point $x$ is in the reduced boundary of $E$.
\item If $x \in \partial^* E$, then in normal coordinates $(V, \psi)$ centered around $x$ we have that $\sigma_{\chi_E}(x) = \nu_{\psi(E)}(\underline{o})$, where $\nu_{\psi(E)}$ is the measure theoretic inner unit normal for $\psi(E)$ and $\underline{o}$ are the coordinates for the center of the chart $x$.
\item If $E, F \subset M$ are sets of finite perimeter, then it holds that for $|D\chi_F|_{\xi}$-a.e.\ point $x \in \partial^*E \cap \partial^* F$ we have $\sigma_E(x) = \langle \sigma_E(x), \sigma_F(x) \rangle_x \sigma_F(x)$.
\end{itemize}
\end{remark}
We also record the following elementary lemma, which can be proved by using Lemma~\ref{lemma:chart_fp} and the analogous statement in the Euclidean setting.
\begin{lemma}\label{lem:compactness_bv}
Let $u_n$ be a sequence of functions in $BV(M)$ such that
\begin{equation*}
\sup_{n \in \mathbf{N}} \int_M |Du_n|_{\xi} < +\infty.
\end{equation*}
Then $\{u_n\}$ is precompact in $L^1(M)$ and every limit point is in $BV(M)$.
\end{lemma}

\subsection{Transportation distance}
Here we recall the definition of $TL^p$-convergence introduced in~\cite{GarciaTrillos2016} and we introduce the notion of weak $TL^p$-convergence. Let $(M, g)$ be a $k$-dimensional compact Riemannian manifold. For a fixed $1 \le p < \infty$ let $\mu, \nu \in \mathcal{P}(M)$, $u \in L^p(\mu), v \in L^p(\nu)$, we set

\begin{equation}
d_{TL^p}((\mu, u), (\nu, v)) := \inf_{\pi \in \Gamma(\mu, \nu)} \left\{ \left(\int_{M \times M} d_M^p(x,y) + |u(x) - u(y)|^p d\pi\right)^{1/p}\right\},\nonumber
\end{equation}
where the infimum is taken over the space of couplings between $\mu$ and $\nu$, which we denote by $\Gamma(\mu, \nu)$. For $p = \infty$,  $\mu, \nu \in \mathcal{P}(M)$, $u \in L^p(\mu), v \in L^p(\nu)$ we set

\begin{equation}
d_{TL^{\infty}}((\mu, u), (\nu, v)) := \inf_{\pi \in \Gamma(\mu, \nu)} \left\{ \operatorname{ess sup}_{x, y \in M}\left( d_M(x, y) + |u(x) - u(y)| \right)\right\}.\nonumber
\end{equation}
We call $d_{TL^p}$ the $TL^p$-metric. It can be shown that $d_{TL^p}$ is a metric on

\begin{equation}
\mathcal{L}^p := \left\{ (\mu, u):\ \mu \in \mathcal{P}(M),\ u \in L^p(\mu)\right\},\nonumber
\end{equation}
this is done in\cite{GarciaTrillos2016} for the Euclidean case, the case of a compact manifold is analogous. Let $\{\pi_n\}_{n} \subset \Gamma(\mu, \nu)$ be a sequence of transport plans between $\mu$ and $\nu$, we say that the these are \emph{$p$-stagnating} if

\begin{equation}
\lim_{n \to +\infty} \int_{M \times M} d_M^p(x, y) d\pi_n = 0.\nonumber
\end{equation}
Transport maps $T_n$ between $\mu, \nu \in \mathcal{P}(M)$ are said to be \emph{$p$-stagnating} if the corresponding transport plans $(Id \times T_n)_{\#}\mu$ are $p$-stagnating. The following propositions are straightforward generalizations of~\cite[Proposition 3.12]{GarciaTrillos2016} and~\cite[Proposition 2.6]{GarciaTrillos2018}.

\begin{proposition}\label{prop:equivalent_tl2}
Let $(\mu_n, u_n), (\mu, u) \in \mathcal{L}^p$, $n \in \mathbf{N}$, $1\le p < +\infty$. Assume that $\mu$ is absolutely continuous with respect to $\volm$. Then the following are equivalent:
\begin{enumerate}[(i)]
\item $(\mu_n, u_n) \to (\mu, u)$ in the $TL^p$ sense.
\item For every sequence of $p$-stagnating transport maps $T_n$ we have
\begin{equation}
\lim_{n \to +\infty} \int_{M} |u_n(T_n(x)) -u(x)|^p d\mu(x) = 0.\nonumber
\end{equation}
\item There exists a sequence of $p$-stagnating transport maps $T_n$ such that
\begin{equation}
\lim_{n \to +\infty} \int_{M} |u_n(T_n(x)) - u(x)|^p d\mu(x) = 0.\nonumber
\end{equation}
\end{enumerate}
\end{proposition}

\begin{proposition}\label{prop:inner_pr}
Suppose that $(\mu_n, u_n) \to (\mu, u)$ in $TL^2(M)$ and $(\mu_n, v_n) \to (\mu, v)$ in $TL^2(M)$. Then
\begin{equation}
\lim_{n\to \infty} \langle u_n, v_n \rangle_{L^2(\mu_n)} = \langle u, v \rangle_{L^2(\mu)}.\nonumber
\end{equation}
\end{proposition}
We will also make use of the following result, which can easily be derived from~\cite[Theorem 2]{GarciaTrillos2020}.

\begin{theorem}\label{thm:transp_plan_existence}
Let $M$ be a $k$-dimensional compact Riemannian submanifold of $\mathbf{R}^d$. Let $\rho \in C^{\infty}(M)$, $\rho > 0$ such that $\nu := \rho \volm \in \mathcal{P}(M)$. Let $\{X_i\}_{i \in \mathbf{N}}$ be iid random points in $M$ distributed according to $\nu$ and let $\nu_n := \frac{1}{n}\sum_{i=1}^n \delta_{X_i}$ be the associated empirical measures. Then there is a constant $C > 0$ such that almost surely there exist transport maps $T_n$ such that $(T_n)_{\#} \nu = \nu_n$ and

\begin{align}\label{eq:cond_transp_plans}
\begin{cases}
\limsup_{n \to +\infty} \frac{n^{1/2}\sup_{x \in M} d_M(x, T_n(x))}{log^{3/4}(n)} \le C \text{if}\ k = 2
\\ \limsup_{n \to +\infty} \frac{n^{1/k}\sup_{x \in M} d_M(x, T_n(x))}{log^{1/k}(n)} \le C \text{if}\ k \ge 3
\end{cases}
\end{align}
\end{theorem}

The correct notion of convergence for obtaining Theorem~\ref{thm:discretenonloc} is weak $TL^2$-convergence, because this is the topology in which we get $\Gamma$-compactness. More generally, let us introduce the notion of \textit{weak $TL^p$-convergence}. 

\begin{definition}\label{def:weaktlp}
Let $\mu$ be a probability measure on $M$ which is absolutely continuous with respect to the volume measure $\volm$, and let $u \in L^p(\mu)$. A sequence $(\mu_n, u_n) \in \mathcal{L}^p$ is said to \textit{converge weakly to $(\mu, u)$ in $TL^p$} if there exists a sequence of $q$-stagnating transport maps $T_n$ between $\mu$ and $\mu_n$ such that the functions $u_n \circ T_n$ converge weakly to $u$ in $L^p(\mu)$. Here, $q = \frac{p}{p-1}$ is the conjugate exponent for $p$. 
\end{definition}
We record the following useful result, which says that the previous definition is independent of the sequence of $q$-stagnating transport maps.

\begin{proposition}\label{prop:weak_tlp}
Let $1 < p < +\infty$ and $q = \frac{p}{p-1}$. Let $\mu$ be a probability measure on $M$ which is absolutely continuous with respect to the volume measure and let $u \in L^p(\mu)$. Assume that $(u_n, \mu_n) \in \mathcal{L}^p$ is a sequence converging weakly to $(u, \mu)$ in $TL^p$. Then for every sequence $S_n$ of $q$-stagnating transport maps between $\mu$ and $\mu_n$ the functions $u_n \circ S_n$ converge weakly to $u$ in $L^p(\mu)$.
\end{proposition}

\begin{proof}
We let $T_n$ the sequence of $q$-stagnating transport maps as in Definition \ref{def:weaktlp}. Let $S_n$ be an arbitrary sequence of $q$-stagnating transport maps between $\mu$ and $\mu_n$. Observe that $\Vert u_n \circ S_n \Vert_{L^p(\mu)} = \Vert u_n \Vert_{L^p(\mu_n)} = \Vert u_n \circ T_n \Vert_{L^p(\mu)}$. In particular, the sequence $u_n \circ S_n$ is bounded in $L^p(\mu)$ and thus, up to extracting a subsequence, we may assume that it converges weakly to a limit $v \in L^p(\mu)$. We need to show that $v = u$. To do so, pick $\varphi \in C^{\infty}(M)$ and observe that by H\"{o}lder's inequality

\begin{align*}
&\left| \int_M u_n \circ T_n \varphi d\mu - \int_M u_n \circ T_n \varphi_{|V_n}\circ T_n d\mu \right| 
\\ &\le \left( \int_M |u_n \circ T_n |^p d\mu\right)^{1/p} \left( \int_M (\operatorname{Lip}\varphi)^q d_M(x, T_n(x))^qd\mu\right)^{1/q},\nonumber
\end{align*}
a similar estimate holds true with $T_n$ replaced by $S_n$. This clearly implies that

\begin{align*}
&\left| \int_M u_n \circ T_n \varphi d\mu - \int_M u_n \circ S_n \varphi d\mu \right| 
\\ & \le C \left( \int_M d_M(x, T_n(x))^qd\mu\right)^{1/q} + C\left( \int_M d_M(x, S_n(x))^qd\mu\right)^{1/q}\nonumber
\end{align*}
The right hand side converges to zero as $n \to +\infty$ because the sequences of transport maps are $q$-stagnating.
\end{proof}

Finally, we have the following natural improvement of Proposition~\ref{prop:inner_pr}.

\begin{proposition}\label{prop:weak_strong}
Let $1 \le p < +\infty$. Let $\mu$ be a probability measure on $M$ which is absolutely continuous with respect to the volume element. Assume that $(u_n, \mu_n) \in \mathcal{L}^p$ converges weakly in $TL^p$ to $(u, \mu)$ and that $(v_n, \mu_n)$ is a sequence in $\mathcal{L}^q$ converging strongly in $TL^q$ to $(v, \mu)$, where $q$ is the conjugate exponent of $p$. Then 
\begin{equation}
\lim_{n\to +\infty} \int_M u_n v_n d\mu_n = \int_M u v d\mu.\nonumber
\end{equation}
\end{proposition}

\subsection{Asymptotics for the graph Laplacian and for the degrees}

Here we recall the results about the convergence of the graph Laplacian contained in~\cite{Coifman2006} and~\cite{Hein2007}. To do so, we need to introduce some definitions. We write $\eta_{\epsilon}(t) := \frac{1}{\epsilon^k} \eta(\frac{t}{\epsilon})$, and we introduce the function $k_{\epsilon}: M \times M \to \mathbf{R}$ defined as $k_{\epsilon}(x,x) = 0,\ x \in M$ and

\begin{equation}
k_{\epsilon}(x,y) = \eta_{\epsilon}(| x - y |_d),\nonumber
\end{equation}
where $|\cdot|_{d}$ denotes the standard Euclidean norm in the ambient space $\mathbf{R}^d$. Observe that 

\begin{equation}\label{eq:unif_deg}
\lim_{\epsilon \downarrow 0} \int_M k_{\epsilon}(x, y) d\nu(y) = C_1\rho(x)\quad \text{uniformly in}\ x \in M.
\end{equation}
Define also 

\begin{equation}
d^{(n, \epsilon)}(x) = \frac{1}{n} \sum_{i=1}^n k_{\epsilon}(x, X_i),\ x \in M,\nonumber
\end{equation}
so that the weights and the degrees on the graph $G_{n, \epsilon_n}$ can be expressed via

\begin{align*}
w_{ij}^{(n, \epsilon)} = k_{\epsilon}(X_i, X_j),\  d^{(n, \epsilon)}_i = d^{(n, \epsilon)}(X_i).
\end{align*}
In this way the random walk graph Laplacian may be extended to an operator acting on functions $f \in C^{\infty}(M)$ as

\begin{equation}
\Delta_{n, \epsilon}f(x) = \frac{1}{\epsilon^2}\left( f(x) - \sum_{j=1}^n\frac{k_{\epsilon}(x, X_j)f(X_j)}{nd^{n, \epsilon}(x)} \right).\nonumber
\end{equation}
Finally, for any $\epsilon > 0$ and any $f \in C^{\infty}(M)$ we define

\begin{equation}
\Delta_{\epsilon}f(x) = \frac{1}{\epsilon^2}\left( f(x) - \frac{\int_M k_{\epsilon}(x, y)f(y)d\nu(y)}{\int_M k_{\epsilon}(x, y) d\nu(y)}\right).\nonumber
\end{equation}
The following theorem is contained in Coifman and Lafon~\cite{Coifman2006}.
\begin{theorem}
Let the assumptions on $M$ in Theorem~\ref{thm:discretenonloc} be in place. For $K \in \mathbf{R}$ define
\begin{equation}
E_K = \{ f \in C^{\infty}(M):\ \Vert f \Vert_{C^3} \le K \}.\nonumber
\end{equation}
Then uniformly in $x$ and uniformly on $E_K$ we have
\begin{equation}
\Delta_{\epsilon}f(x) = \frac{C_2}{2C_1} \Delta_{\rho^2}f(x) + o(\epsilon).\nonumber
\end{equation}
\end{theorem}
One can then apply the previous theorem to obtain the following slight modification of Theorem 28 in~\cite{Hein2007}.
\begin{theorem}\label{thm:conv_lap}
Let the assumptions of Theorem~\ref{thm:discretenonloc} be satisfied. Then with probability one, for all $f \in C^\infty(M)$ we have that $\Delta_{n, \epsilon_n}f \to \frac{C_2}{2C_1}\Delta_{\rho^2}f$ in $TL^2$.
\end{theorem}
In the following, we need also the following simple result about the convergence of the degrees.
\begin{lemma}\label{lem:degrees_conv}
Let the assumptions of Theorem~\ref{thm:discretenonloc} be satisfied. Then it holds with probability one that if $T_n$ is a sequence of transport maps such that 
\begin{equation}
\lim_{n \to +\infty} \sup_{x \in M} d_M(x, T_n(x)) = 0,\nonumber
\end{equation}
then
\begin{equation}\label{eq:linf_conv_degrees}
\lim_{n \to +\infty} \Vert d^{n, \epsilon_n}\circ T_n - C_1 \rho \Vert_{L^{\infty}(\nu)} = 0.
\end{equation}
Moreover, $\lim_{n \to +\infty} \Vert \frac{1}{(d^{n, \epsilon}\circ T_n)^{1/2}} - \frac{1}{(C_1\rho)^{1/2}}\Vert_{L^\infty(\nu)} = 0$.
\end{lemma}
\begin{proof}
The second assertion follows from~\eqref{eq:linf_conv_degrees} and the fact that $\rho > 0$ on the compact manifold $M$. In the following we write $d_n$ for $d^{n, \epsilon}$. To prove~\eqref{eq:linf_conv_degrees} we first observe that

\begin{align}
\Vert d_n \circ T_n - C_1\rho \Vert_{L^{\infty}(\nu)} \le &\ \Vert d_n \circ T_n - C_1\rho \circ T_n \Vert_{L^{\infty}(\nu)} \nonumber
\\&+ \operatorname{Lip}(\rho)\sup_{x \in M} d_M(x, T_n(x)).\nonumber
\end{align}
Thus we only need to prove that $\Vert d_n \circ T_n - C_1\rho \circ T_n\Vert_{L^{\infty}(\nu)}$ converges to zero. To this aim, observe that we may write

\begin{equation}
\Vert d_n \circ T_n - C_1\rho \circ T_n\Vert_{L^{\infty}(\nu)} = \max_{i=1, . . . , n} \left| d_n(X_i) - C_1\rho(X_i) \right|,\nonumber
\end{equation}
and that by the fact that $X_i$'s are identically distributed, if $\gamma > 0$, then

\begin{align*}
\mathbb{P}\bigg(\max_{i = 1, . . . , n} \left| d_n(X_i) - C_1\rho(X_i) \right| \ge \gamma\bigg) &\le n\mathbb{P}\bigg(\left| d_n(X_1) - C_1\rho(X_1) \right| \ge \gamma\bigg)
\\ &= n \int_{M} \mathbb{P}\bigg(|d_n(x) - C_1\rho(x)| \ge \gamma\bigg) d\nu(x).
\end{align*}
Fix $x \in M$,then

\begin{align*}
\mathbb{P}\bigg(|d_n(x) - C_1\rho(x)| \ge \gamma\bigg) \le&\ \mathbb{P}\left(\left| \frac{1}{n} \sum_{j=1}^n k_{\epsilon_n}(x, X_j) - \int_M k_{\epsilon_n}(x, y) d\nu(y)\right| \ge \frac{\gamma}{2}\right)
\\ &+\mathbb{P}\left( \left| \int_M k_{\epsilon_n}(x, y) d\nu(y) - C_1\rho(x)\right| \ge \frac{\gamma}{2}\right).
\end{align*}
The second term on the right hand side is zero for $n$ sufficiently large because of~\eqref{eq:unif_deg}. Thus for $n$ large enough depending on $\gamma$,

\begin{equation}
\mathbb{P}\bigg(|d_n(x) - C_1\rho(x)| \ge \gamma\bigg) \le\ \mathbb{P}\left(\left| \frac{1}{n} \sum_{j=1}^n k_{\epsilon_n}(x, X_j) - \int_M k_{\epsilon_n}(x, y) d\nu(y)\right| \ge \frac{\gamma}{2}\right).\nonumber
\end{equation}
We now proceed at estimating the right hand side. Observe that $Y_j := k_{\epsilon_n}(x, X_j)$ are iid\ random variables with 

\begin{equation}
\mathbb{E}[Y_j] = \int_M k_{\epsilon_n}(x, y)d\nu(y),\ |Y_j| \le \frac{\Vert \eta \Vert_{\infty}}{\epsilon_n^k},\ \operatorname{Var}(Y_j) \le C_1c_0 \frac{\Vert \eta \Vert_ {\infty}}{\epsilon_n^k}.\nonumber
\end{equation}
Thus we can apply Bernstein's inequality to get the following bound

\begin{align*}
&\mathbb{P}\left(\left| \frac{1}{n} \sum_{j=1}^n k_{\epsilon_n}(x, X_j) - \int_M k_{\epsilon_n}(x, y) d\nu(y)\right| \ge \frac{\gamma}{2}\right) 
\\ &\le 2\operatorname{exp}\left( \frac{-n\frac{\gamma^2}{4}\epsilon_n^k}{2\Vert \eta \Vert_{\infty}C_1c_0 + \frac{2}{3}\Vert \eta \Vert_{\infty}\frac{\gamma}{2}} \right)
\end{align*}
Putting things together and summing over $n$ we have that, for some $n(\gamma) \in \mathbf{N}$ depending on $\gamma$

\begin{align*}
&\sum_{n \in \mathbf{N}}\mathbb{P}\bigg(\Vert d_n \circ T_n - C_1\rho \circ T_n \Vert_{L^{\infty}(\nu)} \ge \gamma\bigg) 
\\ &\le n(\gamma) + \sum_{n = n(\gamma) + 1} 2n\operatorname{exp}\left( \frac{-n\frac{\gamma^2}{4}\epsilon_n^k}{2\Vert \eta \Vert_{\infty}C_1c_0 + \frac{2}{3}\Vert \eta \Vert_{\infty}\frac{\gamma}{2}} \right).
\end{align*}
The latter sum is finite if $\frac{n{\epsilon_n^k}}{\log(n)}\to +\infty$. We conclude by Borel-Cantelli's lemma that almost surely $\Vert d_n \circ T_n - C_1\rho \circ T_n\Vert_{L^{\infty}(\nu) }\to 0$, which concludes the proof.
\end{proof}
\subsection{$\Gamma$-convergence of the Dirichlet energies}

In the setting of Section~\ref{sec:mainres}, we define for each random graph $G_{n, \epsilon_n}$ the Dirichlet energy functional $E_n$ as:

\begin{equation}\label{eq:dir_en_graph}
E_n(u) = \frac{1}{2}|\nabla u|_{\mathcal{E}_{n, \epsilon_n}}^2,\quad u \in \mathcal{V}_{n}.
\end{equation}
In the proof of Theorem~\ref{thm:conv_heat} we will need the following result about the $\Gamma$-convergence of the Dirichlet energies defined on the graphs to the Dirichlet energy on the manifold.
\begin{theorem}\label{thm:gamma_dirichlet}
Let $M$ be a $k$-dimensional compact Riemannian manifold embedded in $\mathbf{R}^d$, $k \ge 2$. Let $\rho >0$ be a smooth function on $M$ such that $\nu = \rho \volm \in \mathcal{P}(M)$. If $k \ge 2$ and $\frac{\epsilon_n^kn}{\log(n)} \to +\infty$ as $n \to +\infty$ then
\begin{equation}
E_n \xrightarrow{\Gamma-TL^2} \frac{C_2}{2}E,\nonumber
\end{equation}
where the energy $E$ is defined on $L^2(M)$ as
\begin{equation}\label{eq:dirichlet_energy}
E(u) = \begin{cases}
\frac{1}{2}\int_M |\nabla u|^2 \rho^2 d\volm\ &\text{if}\ u \in H^1(M) \\
+\infty\ &\text{otherwise}.
\end{cases}\nonumber
\end{equation}
Moreover if $u \in C^{\infty}(M)$, then $\limsup_{n \to +\infty} E_{n}(u) \le \frac{C_2}{2} E(u)$. Finally, we have the following compactness property: if $u_n$ are such that

\begin{equation}
\sup_{n \in \mathbf{N}} E_{n}(u_n) < +\infty,\ \sup_{n \in \mathbf{N}} \Vert u_n \Vert_{L^2(M)} < +\infty\nonumber
\end{equation}
then the sequence $u_n$ is precompact in $TL^2$.
\end{theorem}
The proof of Theorem 8 is a straightforward adaptation of the argument given in Theorem 1.4 in~\cite{GarciaTrillos2018} for the flat case. For the sake of completeness, we include a proof in the \hyperlink{sec:appendix}{Appendix}.

\subsection{The optimal energy dissipation inequality}
Let us recall the following notion of weak solution of the heat equation on the weighted manifold $M$.
\begin{definition}
Let $(M, g, \mu = \xi \volm)$ be a weighted $k$-dimensional compact Riemannian submanifold of $\mathbf{R}^d$, with $\xi > 0$ smooth. Let $u_0 \in L^2(\mu)$ and $c > 0$. A weak solution for the diffusion equation
\begin{equation}\label{eq:weak_sol}
\begin{cases}
\partial_t u = -c \Delta_{\xi} u \\
u(x,0) = u_0
\end{cases}
\end{equation}
is a function $u \in L^2_{loc}([0, +\infty), H^1(M))$ such that $u' \in L^2_{loc}([0, +\infty), H^{-1}(M))$ for which
\begin{enumerate}
\item $u(0) = u_0$,
\item $(\partial_t u, \xi w)_{H^1;H^{-1}} + c\int_M g_x\left( \nabla u, \nabla w\right) \xi d\volm = 0$ for a.e. $t$ and all $w \in H^1(M)$.
\end{enumerate}
\end{definition}
The following lemma is a well-known result, which says that the equation~\eqref{eq:weak_sol} is completely characterized by the energy dissipation inequality~\eqref{eq:ener_dis_inq}.

\begin{lemma}\label{lem:energy_diss}
Let $u \in L^2_{loc}([0, +\infty), H^1(M))$ be such that $u' \in L^2_{loc}([0, +\infty), L^2(M))$. Let $u_0 \in C^{\infty}(M)$. Then $u$ is a weak solution of~\eqref{eq:weak_sol} if and only if $u(0) = u_0$ and $u$ satisfies the optimal energy dissipation inequality for a.e.\ $t \in [0, +\infty)$, i.e.,
\begin{align}\label{eq:ener_dis_inq}
\begin{aligned}
cE[u(t)] + \frac{1}{2}\int_0^t \int_{M} c^2|\Delta_\xi u|^2\xi d\volm ds + \frac{1}{2}\int_0^t \int_{M} |u'|^2\xi d\volm ds \le cE[u_0]
\end{aligned}
\end{align}
where we define, for $v \in H^1(M)$,
\begin{equation}
E[v] = \frac{1}{2}\int_M |\nabla v|^2 \xi d\volm.
\end{equation}
\end{lemma}

\begin{proof}
We first observe that whenever $u \in L^2_{loc}([0, +\infty), H^1(M))$ is a function such that $u' \in L^2_{loc}([0, +\infty), L^2(M))$ and $\Delta_{\xi}u \in L^2_{loc}([0, +\infty), L^2(M))$ we have

\begin{equation}\label{eq:time_der_en}
\frac{d}{dt} cE[u(t)] = c \int_M \Delta_{\xi}u \partial_t u\ \xi d\volm.
\end{equation}
Now, assume first that $u$ is a weak solution of~\eqref{eq:weak_sol}. Then by parabolic regularity, $u$ is smooth and~\eqref{eq:time_der_en} is thus true. Using the equation for $\partial_t u$ we obtain that
\begin{align*}
&c \int_M \Delta_{\xi}u \partial_t u \xi d\volm \\ &= -\frac{c^2}{2}\int_M |\Delta_{\xi}u|^2\xi d\volm -\frac{1}{2} \int_M |\partial_t u|^2 \xi d\volm.
\end{align*}
Thus, integrating~\eqref{eq:time_der_en} in time we get the required inequality (actually, equality).

Conversely, assume that~\eqref{eq:ener_dis_inq} is satisfied. Then we clearly infer that $\Delta_{\xi}u \in L^2([0, +\infty), L^2(M))$, we can use~\eqref{eq:time_der_en} in~\eqref{eq:ener_dis_inq} to get, after completing the square,

\begin{align*}
\frac{1}{2}\int_0^t \int_M \left(c\Delta_{\xi}u + u' \right)^2 \xi d\volm ds \le 0.
\end{align*}
which forces $\partial_t u = -c\Delta_{\xi}u$, thus $u$ is a weak solution of~\eqref{eq:weak_sol}.
\end{proof}

In a similar way, one can prove that solutions of the heat equation on a graph also satisfy an energy dissipation inequality. Namely, we have the following result.
\begin{lemma}\label{lem:energy_diss_graph}
Let $G_{n, \epsilon}$ be a graph as constructed in Section~\ref{sec:thescheme}. Let $u_0 \in \mathcal{V}_{n, \epsilon}$. Let $v(x,t) = e^{-t\Delta_{n, \epsilon}}u_0(x)$. Then for all $t \in [0, +\infty)$ the optimal energy dissipation inequality is satisfied, i.e.\ 
\begin{equation}
E_n[v(t)] + \frac{1}{2}\int_0^t |\Delta_{n,\epsilon}v(s)|^2_{\mathcal{V}_{n, \epsilon}}ds+ \frac{1}{2}\int_0^t |\frac{d}{ds}v(s)|^2_{\mathcal{V}_{n, \epsilon}}ds \le E_n[u_0],
\end{equation}
where $E_n$ is the Dirichlet energy defined in~\eqref{eq:dir_en_graph} with $\epsilon_n$ replaced by $\epsilon$.
\end{lemma}

\section{Proofs}\label{sec:proofs}

\subsection{Convergence of the heat operators}
\begin{proof}[Proof of Theorem~\ref{thm:conv_heat}]
We first prove the following result: Assume that $u_0 \in C^{\infty}(M)$, then for every $t>0$ we have

\begin{equation}\label{eq:smooth_claim}
e^{-t\Delta_{n, \epsilon_n}}(u_0|_{V_n}) \to e^{-\frac{C_2}{2C_1}t\Delta_{\rho^2}}u_0\ \text{in}\ TL^2.
\end{equation}

To this aim, observe that it suffices to prove the result for the case $C_1 = 1$; the general case follows by rescaling the weight functions. Define, for $n \in \mathbf{N}, t \ge 0$ and $x \in \mathcal{V}_n(\omega)$,

\begin{equation}
v_n(\omega, x,t) := \left(e^{-t\Delta_{n, \epsilon_n}}u_0|_{V_n(\omega)} \right)(x).\nonumber
\end{equation}
Here $\omega$ is a sample point from some probability space $(\Omega, \mathbb{P})$ on which the random variables $\{X_i\}_{i \in \mathbf{N}}$ are defined. We want to prove that $\mathbb{P}$-a.s.\ for every sequence of $2$-stagnating transport maps $T_n$ from $\nu$ to $\nu_n(\omega) := \frac{1}{n}\sum_{i=1}^n \delta_{X_i(\omega)}$ we have that, for any $t>0$

\begin{equation}\label{eq:claim_explained}
\lim_{n \to +\infty} \int_M | v_n(t,T_n(x)) - u(t, x)|^2 d\nu = 0,
\end{equation}
where $u(t, \cdot) = e^{-\frac{C_2}{2}t\Delta_{\rho^2}}u$. To this aim, pick $\omega \in \Omega$ such that the following conditions are satisfied:

\begin{enumerate}[(i)]
\item There exists a sequence of $2$-stagnating transport maps $T_n$ such that\label{eq:2stagn}~\eqref{eq:cond_transp_plans} is satisfied.
\item $E_n \xrightarrow{\Gamma-TL^2} \frac{C_2}{2}E$.\label{eq:gammaconv}
\item $\Delta_{n, \epsilon_n} f \xrightarrow{TL^2} \frac{C_2}{2C_1}\Delta_{\rho^2}f$ for every $f \in C^{\infty}(M)$.\label{eq:convlap}
\item $\Vert d^{n, \epsilon_n} \circ T_n - C_1 \rho\Vert_{L^{\infty}(\nu)} \to 0$.\label{eq:degconv}
\end{enumerate}
Observe that by Theorem~\ref{thm:transp_plan_existence}, Theorem~\ref{thm:gamma_dirichlet}, Theorem~\ref{thm:conv_lap} and Lemma~\ref{lem:degrees_conv} these conditions hold for $\mathbb{P}$-a.e.\ $\omega \in \Omega$. Thus if we prove~\eqref{eq:claim_explained} for such an $\omega$ we are done. From now on, we will assume $\omega$ to be fixed, so we drop this variable for ease of notation.
Recalling Proposition~\ref{prop:equivalent_tl2},  we just need to show~\eqref{eq:claim_explained} for the sequence of transport maps in~\eqref{eq:2stagn}. Define $\tilde{v}_n(t, x) := v_n(t, T_n(x))$, for $t \in [0, +\infty), x \in M$. By condition~\eqref{eq:degconv} and assuming that $n$ is sufficiently large we have that for all $x \in V_n$,

\begin{equation}
\frac{C_1}{2c} \le d^{n, \epsilon_n}(x) \le 2 C_1c. \nonumber
\end{equation}
Here $c>0$ is a constant such that

\begin{equation}
\frac{1}{c} \le \rho \le c\ \text{on}\ M.\nonumber
\end{equation}
In particular we have that there exists a constant $C > 0$ such that for any $w \in \mathcal{V}_{n, \epsilon_n}$, if $\tilde{w} := w \circ T_n$

\begin{equation}\label{eq:l2_comp}
\frac{1}{C} |w|_{\mathcal{V}_{n, \epsilon_n}} \le \Vert \tilde{w} \Vert_{L^2(\nu)} \le C |w|_{\mathcal{V}_{n, \epsilon_n}}.
\end{equation}

Step 1. We claim that given $T > 0$ there exists a constant $C_T < \infty$ for which

\begin{align}
&\sup_{n \in \mathbf{N}} \Vert \tilde{v}_n \Vert_{L^\infty([0, T], L^2(M))} \le C_T,\label{eq:L2_v}
\\ &\sup_{n \in \mathbf{N}} \Vert \frac{d\tilde{v}_n}{dt} \Vert_{L^\infty([0, T], L^2(M))} \le C_T.\label{eq:L2_deriv}
\end{align}
We start by proving~\eqref{eq:L2_v}. This is easy: Using the equation we obtain
\begin{align*}
\frac{d}{dt} \frac{1}{2} |v_n(t)|^2_{\mathcal{V}_{n, \epsilon_n}} &= -\langle v_n(t), \Delta_{n, \epsilon_n}v_n(t) \rangle_{\mathcal{V}_{n, \epsilon_n}}
\\ &=-\langle \nabla_{n, \epsilon_n} v_n(t), \nabla_{n, \epsilon_n} v_n(t) \rangle_{\mathcal{V}_{n, \epsilon_n}} \le 0.
\end{align*}
Thus after integrating in time and recalling~\eqref{eq:l2_comp} we easily obtain~\eqref{eq:L2_v}, once we recall that $u_0 \in C^{\infty}(M)$.
\\
To show~\eqref{eq:L2_deriv} we notice that $q_n(x,t) := \frac{d}{dt}v_n(x,t)$ is the unique solution to

\begin{equation}
\frac{d}{dt}q_n = -e^{-t\Delta_{n, \epsilon_n}}(\Delta_{n, \epsilon_n}u_0|_{G_n})\nonumber
\end{equation}
In particular, arguing as in the proof of~\eqref{eq:L2_v}, we get an $L^\infty$-bound on the $L^2$-norm of $q_n$. Namely, for any $T > 0$ and a possibly different constant $C_T$

\begin{equation}
\sup_{n \in \mathbf{N}} \Vert \tilde{q}_n \Vert_{L^{\infty}([0, T], L^2(M))} \le C_T.\nonumber
\end{equation}
From this, it is not hard to show that the map $t \to \tilde{v}_n(t)$ is weakly differentiable with derivative given by $t \to \tilde{q}_n(t)$. In particular~\eqref{eq:L2_deriv} follows at once.

Step 2. Compactness.
\\
We may apply Lemma~\ref{lem:energy_diss_graph} to obtain that for any $n \in \mathbf{N}$ and any $t > 0$

\begin{equation}\label{eq:dg_ineq}
E_n[v_n(t)] + \frac{1}{2}\int_0^t |\Delta_{n, \epsilon_n}v_n(s)|_{\mathcal{V}_{n, \epsilon_n}}^2ds + \frac{1}{2}\int_0^t |\frac{d}{ds}v_n(s)|^2_{\mathcal{V}_{n, \epsilon_n}}ds \le E_n[u_0|_{G_n}].
\end{equation}
Fix a time horizon $T > 0$ and a countable dense subset $\{t_p\}_{p \in \mathbf{N}} \subset [0, T]$. From~\eqref{eq:dg_ineq} with $t = t_j$ and the compactness property in Theorem~\ref{thm:gamma_dirichlet} we have that, for each $p \in \mathbf{N}$, the sequence $\tilde{v}_n(t_p)$ is precompact in $L^2(M)$. By a diagonal argument we can thus find a subsequence $n_j$ and functions $u_{t_p} \in L^2(M)$ such that

\begin{equation}
L^2(\nu)-\lim \tilde{v}_{n_j}(t_p) = u_{t_p}\ \forall p \in \mathbf{N}.\nonumber
\end{equation}
We claim that $v_{n_j}(t)$ is a Cauchy sequence in $L^2(\nu)$ for any $t \in [0, T)$. Indeed for $p \in \mathbf{N}$, $j, l \in \mathbf{N}$ using the triangle inequality and~\eqref{eq:L2_deriv} we have
\begin{align*}
&\Vert \tilde{v}_{n_{j+l}}(t) - \tilde{v}_{n_j}(t) \Vert_{L^2(\nu)}
\\ &\le\Vert \tilde{v}_{n_{j+l}}(t) - \tilde{v}_{n_{j+l}}(t_p) \Vert_{L^2(\nu)} + \Vert \tilde{v}_{n_{j+l}}(t_p) - \tilde{v}_{n_j}(t_p) \Vert_{L^2(\nu)} + \Vert \tilde{v}_{n_{j}}(t_p) - \tilde{v}_{n_j}(t) \Vert_{L^2(\nu)}
\\ & \le 2C_T |t-t_p| + \Vert \tilde{v}_{n_{j+l}}(t_p) - \tilde{v}_{n_j}(t_p) \Vert_{L^2(\nu)}.
\end{align*}
Now given $\gamma > 0$ select $p \in \mathbf{N}$ such that $|t-t_p| \le \frac{\gamma}{4C_T}$ and $j \in \mathbf{N}$ such that $\Vert \tilde{v}_{n_{j+l}}(t_p) - \tilde{v}_{n_j}(t_p) \Vert_{L^2(\nu)} \le \frac{\gamma}{2}$ for any $l \in \mathbf{N}$, then

\begin{equation}
\Vert \tilde{v}_{n_{j+l}}(t) - \tilde{v}_{n_j}(t) \Vert_{L^2(\nu)} \le \gamma\nonumber
\end{equation}
whenever $l \in \mathbf{N}$; thus $\tilde{v}_{n_j}(t) \to u_t \in L^2(M)$. Define $u(t) = u_t,\ t \in [0, T]$. We have just proved that

\begin{equation}
L^2(\nu)-\lim_{j \to +\infty} \tilde{v}_{n_j}(t) = u(t)\ \forall t \in [0, T].\nonumber
\end{equation}
We need to show that $u$ is characterized by~\eqref{eq:weak_sol}, this will be shown in Step 4 where we will pass to the limit into \eqref{eq:dg_ineq}. To this aim, we will need to be able to pass to the limit in the second and third terms on the left hand side of \eqref{eq:dg_ineq}. By~\eqref{eq:L2_v} and~\eqref{eq:L2_deriv} over a further, non-relabeled subsequence we have that there exists $v \in L^2([0, T], L^2(\nu))$ with $v' \in L^2([0, T], L^2(\nu))$ such that

\begin{equation}
\tilde{v}_{n_j} \xrightharpoonup{L^2(L^2)} v,\ \frac{d}{dt}\tilde{v}_{n_j} \xrightharpoonup{L^2(L^2)} \frac{d}{dt}v,\nonumber
\end{equation}
and by uniqueness of the limit this implies $u(t) = v(t)$. For later, we also record that

\begin{align}\label{eq:lsc_norm}
\liminf_{j \to +\infty} \int_0^t |\frac{d}{ds}v_{n_j}|^2_{\mathcal{V}_{n_j, \epsilon_{n_j}}} ds &\ge\int_0^t \int_M |\frac{d}{ds}u(s)|^2 \rho^2 d\nu.
\end{align}
This easily follows by the weak lower semicontinuity of the $L^2((0,t), L^2(M))$ norm once we observe that $d^{n_j, \epsilon_{n_j}} \circ T_{n_j} \frac{d}{dt}\tilde{v}^{n_j}$ converges weakly to $\rho u'$ in this space.

Step 3. We claim that $\Delta_{\rho^2}u \in L^2_{loc}((0, +\infty), L^2(\nu))$ and that for every $T>0$
\begin{equation}\label{eq:lsclap}
\liminf_{j \to +\infty}\int_0^T |\Delta_{n_j, \epsilon_{n_j}}v_n|^2_{\mathcal{V}_{{n_j, \epsilon_{n_j}}}}dt \ge \int_0^T c^2 |\Delta_{\rho^2} u|^2\rho^2 dx dt,
\end{equation}
where $c:= \frac{C_2}{2}$. 
\\
To show this we observe that from~\eqref{eq:dg_ineq} we obtain that up to taking a further subsequence, the functions $\widetilde{\Delta_{n_j, \epsilon_{n_j}}v_{n_j}} := \Delta_{n_j, \epsilon_{n_j}}v_{n_j} \circ T_{n_j}$ converge weakly in $L^2([0,T], L^2(\nu))$ to a function $w \in L^2([0,T], L^2(\nu))$. We claim that
\begin{equation}\label{eq:claim_lap}
\frac{1}{c}w = \Delta_{\rho^2} u\ \text{in the sense of distributions}.
\end{equation}
If~\eqref{eq:claim_lap} is true, then~\eqref{eq:lsclap} follows by the lower semicontinuity of the $L^2$-norm. To show~\eqref{eq:claim_lap} we take $f \in C_c^{\infty}((0,+\infty))$ and $g \in C^{\infty}(M)$. Then using Theorem~\ref{thm:conv_lap}, Lemma~\ref{lem:degrees_conv}, the convergence of the functions $\tilde{v}_{n_j}$ and using the fact that the Laplacian on the graph is self-adjoint we have
\begin{align*}
&\int_0^{+\infty}f(t)\int_M u(t)\left(\Delta_{\rho^2}g\right) \rho^2 d\volm dt 
\\ &= \lim_{j \to +\infty} \frac{1}{c}\int_0^{+\infty} f(t) \int_M \tilde{v}_{n_j}(t) \widetilde{\Delta_{n_j, \epsilon_{n_j}} g} \tilde{d}_{n_j} d\nu dt
\\ &= \lim_{j \to +\infty} \frac{1}{c} \int_0^{+\infty} f(t) \langle v_{n_j}, \Delta_{n_j, \epsilon_j}g \rangle_{\mathcal{V}_{n_j, \epsilon_{n_j}} }dt
\\ &= \lim_{j \to +\infty} \frac{1}{c} \int_0^{+\infty} f(t) \langle \Delta_{n_j, \epsilon_{n_j}}v_{n_j}, g \rangle_{\mathcal{V}_{n_j, \epsilon_{n_j}}}dt
\\ &= \lim_{j \to +\infty} \frac{1}{c} \int_0^{+\infty} f(t) \int_M \widetilde{\Delta_{n_j, \epsilon_{n_j}}v_{n_j}} \tilde{g}_{n_j} \tilde{d}_{n_j} d\nu dt.
\end{align*}
Observing that $\tilde{g}_{n_j} := g \circ T_{n_j}$ converges uniformly to $g$ we get that the last limit equals
\begin{equation*}
\frac{1}{c} \int_0^{+\infty} f(t) \int_M w(t) g \rho^2 d\volm dt.
\end{equation*}
In particular for every $f \in C^{\infty}_c((0, +\infty))$ and $g \in C^{\infty}(M)$ we have
\begin{equation*}
\int_0^{+\infty}f(t)\int_M u(t)\Delta_{\rho^2}g \rho^2 d\volm dt  = \frac{1}{c} \int_0^{+\infty} f(t) \int_M w(t) g \rho^2 d\volm dt,
\end{equation*}
which clearly gives~\eqref{eq:claim_lap}.

Step 4. Proof of~\eqref{eq:smooth_claim}.
 \\
 Using Theorem~\ref{thm:gamma_dirichlet}, the weak lower semicontinuity~\eqref{eq:lsc_norm} and the lower bound obtained in Step 3 we can pass to the limit in~\eqref{eq:dg_ineq} to get that for all times $t > 0$
 
\begin{align*}
\begin{aligned}
\frac{C_2}{2}E[u(t)] + \frac{1}{2}\int_0^t \int_{M} \left(\frac{C_2}{2}\right)^2|\Delta_{\rho^2}u|^2\rho^2dx ds + \frac{1}{2}\int_0^t \int_{M} |u'|^2\rho^2dx ds \le \frac{C_2}{2}E[u_0]
\end{aligned}
\end{align*}
In particular, applying Lemma~\ref{lem:energy_diss} with $c = C_2 / 2$ we see that $u$ is the unique solution to 

\begin{equation}
\begin{cases}
\partial_t u = -\frac{C_2}{2}\Delta_{\rho^2}u \\
u(0) = u_0
\end{cases}\nonumber
\end{equation}
which, in particular, implies that the limit is independent of the chosen subsequence, thus the whole sequence converges to $u$, as claimed.

Step 5. Conclusion.
\\
Let $T_n$ be a sequence of transportation maps obtained by applying Theorem~\ref{thm:transp_plan_existence}. By Proposition~\ref{prop:equivalent_tl2} we just need to show that for any $t>0$ the functions $e^{-t\Delta_{n, \epsilon_n}}u_n \circ T_n$ converge strongly to $e^{-t\frac{C_2}{2C_1}\Delta_{\rho^2}}u$ in $L^2(M)$. For $s > 0$ define $v_n(x, s) = e^{-s\Delta_{n, \epsilon_n}}u_n$. By differentiating the norm we have

\begin{equation*}
\frac{d}{ds} | v_{n} |_{\mathcal{V}_{n, \epsilon_n}}^2 = - |\nabla_{n} v_n|_{\mathcal{E}_{n, \epsilon_n}}^2.
\end{equation*}
Thus after integrating in $s$ and by using Fatou's Lemma we have that for a fixed $t > 0$

\begin{equation*}
\int_0^t \liminf_{n \to +\infty} |\nabla_n v_n(s)|_{\mathcal{E}_{n, \epsilon_n}}^2 ds \le C\Vert u \Vert_{L^2(\nu)}^2.
\end{equation*}
In particular there exist $0 < s < t$ and a subsequence $n_j$ such that
\begin{equation*}
\sup_{j \in \mathbf{N}} E_{n_j}[v_{n_j}(s)] < +\infty.
\end{equation*}
By Lemma~\ref{lem:energy_diss_graph} applied to $w_{n_j}(r) := e^{-r\Delta_{n_j, \epsilon_{n_j}}}v_{n_j}(s)$ we infer that
\begin{equation*}
\sup_{j \in \mathbf{N}} E_{n_j}[v_{n_j}(t)]  = \sup_{j \in \mathbf{N}} E_{n_j}[w_{n_j}(t-s)] < +\infty.
\end{equation*}
In particular, by the compactness statement of Theorem~\ref{thm:gamma_dirichlet} we obtain that, upon taking a further subsequence, the functions $v_{n_j}(t)$ converge strongly in $TL^2(M)$ to a function $v(t) \in L^2(M)$. We claim that $v = e^{-t\frac{C_2}{2C_1}\Delta_{\rho^2}}u$. To see this, let $g \in C^{\infty}(M)$. Setting $\tilde{v}_{n_j}(t) = v_{n_j}(t) \circ T_{n_j}$ and $\tilde{g}_{n_j} = g \circ T_{n_j}$,  using the fact that $\tilde{g}_{n_j}$ converges uniformly to $g$ and the fact that $e^{-t\Delta_{n, \epsilon_n}}$ are self-adjoint operators we get
\begin{align*}
\int_M v g \rho^2 d\volm &= \lim_{j\to +\infty} \int_M \tilde{v}_{n_j}(t) \tilde{g} \tilde{d}_{n_j} d\nu
\\ &= \lim_{j \to +\infty} \langle v_{n_j}(t), g \rangle_{\mathcal{V}_{n_j, \epsilon_{n_j}}}
\\ &= \lim_{j \to +\infty} \langle u_{n_j}, e^{-t\Delta_{n_j, \epsilon_{n_j}}}g \rangle_{\mathcal{V}_{n_j, \epsilon_{n_j}}}
\\ &=  \lim_{j \to +\infty} \int_M \tilde{u}_{n_j} \widetilde{e^{-t\Delta_{n_j, \epsilon_{n_j}}}g} \tilde{d}_{n_j}d\nu
\\ &= \int_M u e^{-t\frac{C_2}{2C_1}\Delta_{\rho^2}}g \rho^2 d\volm,
\end{align*}
where in the last step we used~\eqref{eq:smooth_claim}. Using the self-adjointness of the heat semigroup on $M$ we infer that for any smooth function $g \in C^{\infty}(M)$,

\begin{align*}
\int_M v g \rho^2 d\volm = \int_M e^{-t\frac{C_2}{2C_1}\Delta_{\rho^2}}u g \rho^2 d\volm.
\end{align*}
Thus $v = e^{-t\frac{C_2}{2C_1}\Delta_{\rho^2}}u$. In particular, the limit does not depend on the chosen subsequence, thus we obtain the claim.
\end{proof}

\subsection{Discrete-to-nonlocal}
\begin{proof}[Proof of Theorem~\ref{thm:discretenonloc}]
The proof follows from Theorem~\ref{thm:conv_heat}, Proposition~\ref{prop:weak_tlp} and the convergence of the degrees in Lemma~\ref{lem:degrees_conv}. The precompactness statement is a consequence of the general fact that bounded sets in $L^2$ are weakly precompact.
\end{proof}

\subsection{Bertozzi's question}

\begin{proof}[Proof of Corollary~\ref{cor:bertozzi_qst}]
By Theorem~\ref{thm:discretenonloc} we know that almost surely, for each $h > 0$,
\begin{equation}
\Gamma(TL^2(M)-\text{weak})-\lim_{n \to +\infty} E^h_{n, \epsilon_n} = \sqrt{\frac{C_1C_2}{2}}E_{\frac{C_2 h}{2C_1}}.
\end{equation}
By the same argument used in the proof of Theorem~\ref{thm:discretenonloc} we have that almost surely, for every $h > 0$, the sequence of energies
\begin{equation*}
D_{n, \epsilon_n}^h(u) := \frac{1}{\sqrt{h}}\sum_{i \neq j} \sigma_{ij}\langle u^i - \chi_n^i , e^{-h\Delta_{n, \epsilon_n}}(u^j - \chi_n^j)\rangle_{\mathcal{V}_n},\quad u \in \mathcal{M}_n,
\end{equation*}
$\Gamma$-converges in the weak-$TL^2(M)$ topology to the energy
\begin{equation*}
D_h(u) = \sqrt{\frac{C_1C_2}{2}}\frac{1}{\sqrt{\frac{C_2h}{2C_1}}}\sum_{i \neq j} \int_M (u^i - \chi^j) e^{-h\frac{C_2}{2C_1}\Delta_{\rho^2}}(u^j - \chi^j)\rho^2d\volm,\quad u \in \mathcal{M}.
\end{equation*}
In particular for every $h > 0$ we have
\begin{equation*}
\Gamma(TL^2(M)-\text{weak})-\lim_{n \to +\infty} (E_{n, \epsilon_n}^h - D_{n, \epsilon_n}^h) = \sqrt{\frac{C_1C_2}{2}}E_{\frac{C_2h}{2C_1}} - D_h.
\end{equation*}
This yields that the minimizers of $E_{n, \epsilon_n}^h - D_{n, \epsilon_n}^h$ converge weakly in $TL^2(M)$ to minimizers of $\sqrt{\frac{C_1C_2}{2}}E_{\frac{C_2h}{2C_1}} - D_h$. The conclusion is then a consequence of the minimizing movements interpretations in Lemma~\ref{lem:dissipationMBO} and Lemma~\ref{lem:minmovman}.
\end{proof}

\subsection{Nonlocal-to-local}
\begin{proof}[Proof of Theorem~\ref{thm:nonlocloc}]
$\Gamma$-$\limsup$. The $\Gamma$-$\limsup$ inequality is a consequence of the consistency part of the theorem, namely that for every $u \in BV(M, \{0,1\}^P) \cap \mathcal{M}$

\begin{equation}\label{eq:consistency_claim}
\lim_{h \downarrow 0} E_h(u) = E(u).
\end{equation}
It is clear that~\eqref{eq:consistency_claim} is a consequence of the following claim: Assume that $E, F \subset M$ are sets of finite perimeter, then

\begin{equation}\label{eq:consistency_two_sets}
\lim_{h \downarrow 0}\frac{1}{\sqrt{h}} \int_M \chi_F \left( \chi_E - e^{-h\Delta_\xi}\chi_E\right) d\mu = \frac{1}{\sqrt{\pi}}\int_{\partial^* E \cap \partial^* F} \langle \sigma_E(x), \sigma_F(x) \rangle_x |D\chi_F|_{\xi}(x).
\end{equation}
Indeed, simply apply~\eqref{eq:consistency_two_sets} to $E = \{u^i = 1\}$, $F=\{u^j = 1\}$, multiply by $\sigma_{ij}$ and sum over all pairs $i \neq j$, since then $\langle \sigma_E(x), \sigma_F(x) \rangle_x = -1$ on $\partial^*E \cap \partial^* F$. We now prove~\eqref{eq:consistency_two_sets} in four steps.

Step 1. Recalling the notation~\eqref{eq:heat_semig}, we can rewrite 

\begin{equation*}
T(h)\chi_E(x) = \chi_E(x) + \int_0^h \frac{d}{dt}T(t)\chi_E(x) dt.
\end{equation*}
Using Theorem~\ref{thm:finite_per_man}  we obtain that the argument of the limit in~\eqref{eq:consistency_two_sets} is equal to

\begin{align}
&\frac{1}{\sqrt{h}}\int_M \chi_F \int_0^h \Delta_\xi T(t)\chi_E(x)dt d\mu(x) \nonumber
\\ &= \frac{1}{\sqrt{h}}\int_{\partial^*F} \langle \sigma_{\chi_F}(x), \int_0^h\nabla T(t)\chi_E(x)dt\rangle d|D\chi_F|_{\xi}(x).\nonumber
\end{align}
Thus by the discussion in Remark~\ref{rem:red_bdry} it suffices to show that for every $x \in \partial^* E \cap \partial^*F$ such that $\sigma_E(x) = \langle \sigma_E(x), \sigma_F(x) \rangle_x \sigma_F(x)$ we have

\begin{equation}\label{eq:first_red}
\frac{1}{\sqrt{\pi}} \langle\sigma_F(x), \sigma_E(x)\rangle = \lim_{h\downarrow 0} \bigg\langle \sigma_F(x), \frac{1}{\sqrt{h}}\int_0^h \nabla T(t)\chi_E(x) dt \bigg\rangle.
\end{equation}

Step 2. We claim that for $s < 1/2$, equation~\eqref{eq:first_red} is equivalent to 

\begin{equation}\label{eq:second_red}
\frac{1}{\sqrt{\pi}} \langle\sigma_F(x), \sigma_E(x)\rangle = \lim_{h\downarrow 0} \bigg\langle \sigma_F(x), \frac{1}{\sqrt{h}}\int_0^h \nabla T(t)(\chi_{E \cap B_{h^s}(x)}(\cdot))(x) dt \bigg\rangle.
\end{equation}

To prove this equivalence, we fix $s < 1/2$ and use~\eqref{eq:gauss_up_bds_II} to show

\begin{equation}\label{eq:bound_g}
\lim_{h \downarrow 0} \frac{1}{\sqrt{h}}\int_0^h \int_{M\setminus B_{h^{s}}(x)} |\nabla_x p(t,x,y)|d\mu(y) = 0.
\end{equation}
Clearly~\eqref{eq:bound_g} then implies the equivalence between~\eqref{eq:first_red} and~\eqref{eq:second_red}. For $j \in \mathbf{N}$ and $t < h$ we denote by $B_j$ the ball $B_{2^{j}t^{s}}(x)$. Observe that $M\setminus B_{h^s}(x) \subset M\setminus B_{t^{s}}(x)$. To prove~\eqref{eq:bound_g} we use the Gaussian upper bound~\eqref{eq:gauss_up_bds_II} to estimate

\begin{align}
&\frac{1}{\sqrt{h}}\int_0^h \int_{M\setminus B_{h^{s}}(x)} |\nabla p(t,x,y)|d\mu(y) dt\nonumber
\\ &\le \sum_{j=0}^{\left[ \text{diam}(M)\right]}\frac{1}{\sqrt{h}}\int_0^h \int_{B_{j+1}\setminus B_{j}} \frac{\hat{C}_1}{\sqrt{t}\mu(B_{\sqrt{t}}(x))}\operatorname{exp}\left(\frac{-d^2(x, y)}{\hat{C}_2t}\right)d\mu(y)dt\nonumber
\\ &\le  \sum_{j=0}^{\left[ \text{diam}(M)\right]}\frac{\hat{C}_1}{\sqrt{h}}\int_0^h \frac{\mu(B_{j+1})}{\sqrt{t}\mu(B_{\sqrt{t}}(x))}\operatorname{exp}\left(-\frac{2^{2j}}{\hat{C}_2t^{1 - 2s}}\right)dt.\label{eq:quoti_balls}
\end{align}
Observe that the doubling property~\eqref{eq:doubling} gives $\frac{\mu(B_{j+1})}{\mu(B_{\sqrt{t}}(x))} \le3 \frac{2^{jN}t^{sN}}{t^{N/2}}$. Thus~\eqref{eq:quoti_balls} is estimated by

\begin{align*}
&\hat{C}_1 \sum_{j=0}^{[\operatorname{diam}(M)]} \frac{1}{\sqrt{h}} \int_0^h \frac{2^{jN}t^{Ns}}{\sqrt{t}t^{N/2}}\operatorname{exp}\left( -\frac{2^{2j}}{\hat{C}_2t^{1-2s}} \right)dt
\\ & = \hat{C}_1 \sum_{j=0}^{[\operatorname{diam}(M)]} 2^{jN} \int_0^h t^{-N/2 + sN - 1} \operatorname{exp}\left( -\frac{2^{2j}}{\hat{C}_2t^{1-2s}} \right)dt
\\ & \le \hat{C}_1 \sum_{j=0}^{[\operatorname{diam}(M)]} 2^{jN} \operatorname{exp}\left( -\frac{2^{2j}}{2\hat{C}_2h^{1-2s}}\right) \int_0^h t^{-N/2 + sN -1}\operatorname{exp}\left( -\frac{1}{2\hat{C}_2t^{1-2s}}\right)dt,
\end{align*}
which converges to zero as $h \downarrow 0$, since the integrand is uniformly bounded and the prefactor converges to zero as $h \downarrow 0$.

Step 3. We claim that 

\begin{align}\label{eq:claim_gauss_sub}
&\lim_{h \downarrow 0} \frac{1}{\sqrt{h}} \bigg\langle \sigma_F(x), \frac{1}{\sqrt{h}}\int_0^h \int_M \nabla_x p(t, x, y)\chi_{E \cap B_{h^s}(x)}(y) d\mu(y)\bigg\rangle
\\ &= \lim_{h\downarrow 0} \frac{1}{\sqrt{h}} \bigg\langle \sigma_F(x), \int_0^h\int_M \nabla_x \left( \frac{e^{-\frac{d(x,y)^2}{4t}}}{(4\pi t)^{k/2}}v_0(x,y)\right)\chi_{E \cap B_{h^s}(x)}(y) d\mu(y)  dt\bigg\rangle,\nonumber
\end{align}
where $v_0$ is the coefficient in the asymptotic expansion~\eqref{eq:asym_exp}.

To see this, observe that~\eqref{eq:asym_exp} applied with $l=1$ and some $N > \frac{k}{2} + l$ yields

\begin{align*}
&\lim_{h \downarrow 0} \frac{1}{\sqrt{h}} \bigg\langle \sigma_F(x), \frac{1}{\sqrt{h}}\int_0^h \int_M \nabla_x p(t, x, y)\chi_{E \cap B_{h^s}(x)}(y) d\mu(y)\bigg\rangle
\\ &= \lim_{h\downarrow 0} \frac{1}{\sqrt{h}} \sum_{j=0}^N\bigg\langle \sigma_F(x), \int_0^h\int_M \nabla_x \left( \frac{e^{-\frac{d(x,y)^2}{4t}}}{(4\pi t)^{k/2}}v_j(x,y)t^j\right)\chi_{E \cap B_{h^s}(x)}(y) d\mu(y) dt \bigg\rangle.
\end{align*}
Thus, all we need to show is that the limit as $h\downarrow 0$ of the terms on the right hand side corresponding to $j \ge 1$ vanishes, i.e., that for $j \ge 1$

\begin{equation}\label{eq:jge_vanish}
\lim_{h\downarrow 0} \frac{1}{\sqrt{h}}\bigg\langle \sigma_F(x), \int_0^h\int_M \nabla_x \left( \frac{e^{-\frac{d(x,y)^2}{4t}}}{(4\pi t)^{k/2}}v_j(x,y)t^j\right)\chi_{E \cap B_{h^s}(x)}(y) d\mu(y) \bigg\rangle = 0.
\end{equation}
To verify~\eqref{eq:jge_vanish}, we compute the argument in the limit in normal coordinates around $x$. Let $\Psi: B_R(x) \to B_R(\underline{o})$ be normal coordinates around $x$. Then $g_{ij}(\underline{o}) = \delta_{ij}$ is the identity matrix and $d^2(x, y) = |\Psi(y)|_k$. Writing $\nu_F$ for the vector of coordinates of $\sigma_F(x)$,  the argument of the limit may be written as

\begin{equation*}
\frac{1}{\sqrt{h}}\int_0^h \int_{B_{h^s}(x)} \nu_F \cdot \left( \frac{-z}{2t}v_j(x, \Phi(z)) + Dv_j(x, \Phi(z))\right) \frac{e^{-\frac{|z|^2}{4t}}}{(4\pi t)^{k/2}}t^j \chi_{\Psi(E)}(z)\gamma(z) dz dt,
\end{equation*}
where we set $\gamma(z) := \sqrt{\operatorname{det}(g)(z)} \rho(\Phi(z))$, with $\Phi = \Psi^{-1}$. By the smoothness of the coefficients $v_j$, by the compactness of the manifold and by the fact that $j \ge 1$, if $h < 1$ we can bound this integral by

\begin{equation}\label{eq:boiundlgeone}
Ch^{j-1} \frac{1}{\sqrt{h}}\int_0^h \int_{\mathbf{R}^k} \left( \frac{1}{2} + 1 \right) \frac{e^{-\frac{|z|^2}{4t}}}{(4\pi t)^{k/2}} dz dt \le C\sqrt{h},
\end{equation}
where $C$ is a constant depending on $v_j$, $M$ and $\xi$. Thus we have ~\eqref{eq:jge_vanish}.

Step 4. Conclusion. We now compute the limit on the right-hand side of~\eqref{eq:claim_gauss_sub}. As before, we work in normal coordinates centered at $x$. With the same notation as in Step 3, the argument of the limit may be rewritten as

\begin{align*}
&\frac{1}{\sqrt{h}} \bigg\langle \sigma_F(x), \int_0^h\int_M \nabla_x \left( \frac{e^{-\frac{d(x,y)^2}{4t}}}{(4\pi t)^{k/2}}v_0(x,y)\right)\chi_{E \cap B_{h^s}(x)}(y) d\mu(y)  dt\bigg\rangle
\\ & = \frac{1}{\sqrt{h}}\int_0^h \int_{B_{h^s}(x)} \nu_F \cdot \left( \frac{-z}{2t}v_0(x, \Phi(z)) + Dv_0(x, \Phi(z))\right) \frac{e^{-\frac{|z|^2}{4t}}}{(4\pi t)^{k/2}} \chi_{\Psi(E)}(z)\gamma(z) dz dt.
\end{align*}
As in Step 3, one can show that 

\begin{equation*}
\lim_{h \downarrow 0} \frac{1}{\sqrt{h}}\int_0^h \int_{B_{h^s}(x)} \nu_F \cdot Dv_0(x, \Phi(z))\frac{e^{-\frac{|z|^2}{4t}}}{(4\pi t)^{k/2}} \chi_{\Psi(E)}(z)\gamma(z) dz dt = 0.
\end{equation*}
Thus, all we need to show is that 

\begin{equation*}
\lim_{h \downarrow 0} -\frac{1}{\sqrt{h}}\int_0^h \int_{B_{h^s}(\underline{o})} \nu_F \cdot \frac{z}{2t}v_0(x, \Phi(z)) \frac{e^{-\frac{|z|^2}{4t}}}{(4\pi t)^{k/2}} \chi_{\Psi(E)}(z)\gamma(z) dz dt = \frac{1}{\sqrt{\pi}}\nu_E \cdot \nu_F.
\end{equation*}
This is essentially already done in~\cite{Miranda2007a}. We sketch the short argument for completeness. After a change of variables in space and time, the argument of the limit may be written as

\begin{equation*}
-\int_0^1 \frac{1}{\sqrt{t}}\int_{B_{\frac{h^s}{\sqrt{ht}}(\underline{o})}} \nu_F \cdot \frac{z}{2}v_0(x, \Phi(\sqrt{ht}z))\frac{e^{-\frac{|z|^2}{4}}}{(4\pi)^{k/2}}\chi_{\frac{\Psi(E)}{\sqrt{ht}}}\gamma(\sqrt{ht}z) dz dt.
\end{equation*}
By De Giorgi's structure theorem we have

\begin{equation}
L^1_{loc}-\lim_{h \downarrow 0} \chi_{\Psi(E)}({\sqrt{ht}}) = \chi_{H_{\nu_E}}\ \forall t \in [0,1],\nonumber
\end{equation}
where $H_{\nu_E}$ is the half space given by
\begin{equation}
H_{\nu_E} := \left\{ z \in \mathbf{R}^k:\ \nu_E \cdot z \le 0 \right\}.\nonumber
\end{equation}
By an application of the dominated convergence theorem we infer that on any compact set $K \subset \mathbf{R}^k$,
\begin{equation}
\lim_{h \downarrow 0} \int_0^1 \int_K |\chi_{\frac{\Psi(E)}{\sqrt{ht}}} - \chi _{H_{\nu_E}}| dz dt = 0.\nonumber
\end{equation}
In particular, upon taking a subsequence, we may assume that

\begin{equation}
\lim_{h \downarrow 0} \chi_{\frac{\Psi(E)}{\sqrt{ht}}}(z) = \chi_{H_{\nu_E}}(z)\ \text{for a.e.}\ (t,z) \in [0,1] \times \mathbf{R}^k.\nonumber
\end{equation}
Moreover we have that

\begin{equation*}
\lim_{h \downarrow 0} v_0(x, \Phi(\sqrt{ht}z))\gamma(\sqrt{ht}z) = 1,\ \text{uniformly in}\ t \in [0,1].
\end{equation*}
Thus by an application of the dominated convergence theorem we get
\begin{align*}
&\lim_{h \downarrow 0} -\int_0^1 \frac{1}{\sqrt{t}}\int_{B_{\frac{h^s}{\sqrt{ht}}(\underline{o})}} \nu_F \cdot \frac{z}{2}v_0(x, \Phi(\sqrt{ht}z))\frac{e^{-\frac{|z|^2}{4}}}{(4\pi)^{k/2}}\chi_{\frac{\Psi(E)}{\sqrt{ht}}}\gamma(\sqrt{ht}z) dz dt
\\ &= -\int_0^1 \frac{1}{2\sqrt{t}} \int_{\nu_E \cdot y \le 0} (y \cdot \nu_F) G_1(z) dz 
\\ & = \int_0^1 \frac{1}{2\sqrt{t}} \int_{\nu_E \cdot y \le 0} (\nu_F \cdot \nu_E) (y \cdot \nu_E)_{-} G_1(z) dz 
\\ & = \frac{1}{\sqrt{\pi}}(\nu_F \cdot \nu_E).
\end{align*}

$\Gamma$-$\liminf$. To prove the $\Gamma$-$\liminf$ inequality we use the blow-up method of Fonseca and M\"{u}ller~\cite{Fonseca1993} (see also~\cite{Alberti1998} and~\cite{Ambrosio2011}). 
\\
Given $u_h \in \mathcal{M}$ such that $u^h \to u \in \mathcal{M}$ in $L^1(M)$, we want to prove that for every sequence $h_n \downarrow 0$
\begin{equation}\label{eq:claim_gliminf_expl}
\liminf_{n \to +\infty} E_{h_n}(u_{h_n})  \ge E(u).
\end{equation}
Clearly, we may without loss of generality assume that the left hand side of~\eqref{eq:claim_gliminf_expl} is finite.

Step 1. $u \in BV(M, \{0,1\}^P)$. 

By Lemma~\ref{lemma:chart_fp} we just need to show that $u \circ \psi$ is in $BV(\psi(V))$ for every chart $(V, \psi)$ of $M$. It is clear that one can restrict to the case when $V = B_r(x_0)$, $r \le R < \frac{\operatorname{inj}(M)}{2}$, $x_0 \in M$ for some fixed $R$ and $\psi = \operatorname{exp}_{x_0}^{-1}$. The statement for a general chart then follows by compactness. So we fix $V = B_r(x_0)$ and $\psi = \operatorname{exp}_{x_0}^{-1}$. We observe that if $N \ge \frac{k}{2}$, by the asymptotic expansion for the heat kernel~\eqref{eq:asym_exp} with $l= 0$ and $t=h$ we get

\begin{align*}
E_{h_n}(u_{h_n}) &\ge \frac{1}{\sqrt{h_n}}\sum_{i,j}\sigma_{ij}\int_{B_r(x_0)} u_{h_n}^i e^{-h_n \Delta_{\xi}}u_{h_n}^j d\mu
\\ &\ge  \frac{1}{\sqrt{h_n}}\sum_{i,j}\sigma_{ij}\int_{B_r(x_0)} u_{h_n}^i(x) \int_{B_r(x_0)} p(h_n, x, y) u_{h_n}^j(y)d\mu(y) d\mu(x)
\\ &\begin{aligned}\ge &\sum_{l=0}^N \frac{1}{\sqrt{h_n}}\sum_{i,j}\sigma_{ij}\int_{B_r(x_0)} u_{h_n}^i(x) \int_{B_r(x_0)} \frac{e^{\frac{-d^2(x,y)}{4h_n}}}{(4\pi h_n)^{k/2}} v_l(x,y)h_n^lu_{h_n}^j(y)d\mu(y) d\mu(x)
\\ &-C_N\sqrt{h_n}
\end{aligned}
\end{align*}
If $l \ge 1$, with an estimate similar to the one used in~\eqref{eq:boiundlgeone} we obtain that

\begin{equation*}
\left| \frac{1}{\sqrt{h_n}}\sum_{i,j}\sigma_{ij}\int_{B_r(x_0)} u_{h_n}^i(x) \int_{B_r(x_0)} \frac{e^{\frac{-d^2(x,y)}{4h_n}}}{(4\pi h_n)^{k/2}} v_l(x,y)h_n^lu_{h_n}^j(y)d\mu(y) d\mu(x)\right| \le C\sqrt{h_n},
\end{equation*}
where $C$ depends on $\sigma, v_l, M$ and $\xi$. Thus we have

\begin{align*}
E_{h_n}(u_{h_n}) \ge &\frac{1}{\sqrt{h_n}}\sum_{i,j}\sigma_{ij}\int_{B_r(x_0)} u_{h_n}^i(x) \int_{B_r(x_0)} \frac{e^{\frac{-d^2(x,y)}{4h_n}}}{(4\pi h_n)^{k/2}} v_0(x,y)u_{h_n}^j(y)d\mu(y) d\mu(x)
\\ &-C\sqrt{h_n}.
\end{align*}
The first term on the right hand side may be rewritten in local coordinates as

\begin{equation}\label{eq:toest}
\frac{1}{\sqrt{h_n}}\sum_{i,j}\sigma_{ij}\int_{B_r(\underline{o})}\tilde{u}_{h_n}^i(x) \int_{B_r(\underline{o})} \frac{e^{\frac{-d^2(\psi^{-1}(x),\psi^{-1}(y))}{4h_n}}}{(4\pi h_n)^{k/2}} \tilde{v}_0(x,y)\tilde{u}_{h_n}^j(y)\gamma(y)dy\ \gamma(x)dx,
\end{equation}
with $\gamma(x) = \sqrt{\operatorname{det}(g)} \xi(\psi^{-1}(x))$, $\tilde{v}_0(x,y) = v_0(\psi^{-1}(x), \psi^{-1}(y)) $ and $\tilde{u} = u \circ \psi^{-1}$. Let $L$ be such that $d(\psi^{-1}(x), \psi^{-1}(y)) \le L |x-y|_k$. Then~\eqref{eq:toest} may be bounded from below by

\begin{equation*}
\frac{\inf_{x, y \in B_{r}(\underline{o})} \left\{ \tilde{v}_0(x,y)\gamma(y)\gamma(x)\right\}}{L^{k+1}}E^{euclid}_{\frac{h_n}{L^2}}(1_{B_r(\underline{o})}\tilde{u}_{h_n}),
\end{equation*}
where we set

\begin{equation*}
E^{euclid}_{\frac{h_n}{L^2}}(1_{B_r(\underline{o})}\tilde{u}_{h_n}) = \frac{1}{\sqrt{\frac{h_n}{L^2}}}\sum_{i,j}\sigma_{ij}\int_{\mathbf{R}^k}\tilde{u}_{h_n}^i(x) \int_{\mathbf{R}^k} \frac{e^{\frac{-L^2|x-y|^2}{4h_n}}}{(4\pi \frac{h_n}{L^2})^{k/2}} \tilde{u}_{h_n}^jdy\ dx.
\end{equation*}
In particular we obtain that

\begin{equation*}
+\infty > \liminf_{n \to +\infty} E^{euclid}_{\frac{h_n}{L^2}}(\mathbf{1}_{B_r(\underline{o})}\tilde{u}_{h_n}),
\end{equation*}
which says that $\mathbf{1}_{B_r(\underline{o})}\tilde{u}$ is in $BV(\mathbf{R}^k, \{0,1\}^P)$ by an application of, for example, Lemma A.4 in~\cite{Esedoglu2015}.

Step 2. We now turn to~\eqref{eq:claim_gliminf_expl}. By Step 1 we know that $u \in BV(M, \{0,1\}^P)$. We set $\Omega_i := \{ u^i = 1\}$. Passing to a subsequence if necessary, we may assume that

\begin{equation}\label{eq:limisliminf}
\lim_{n \to +\infty} E_{h_n}(u_{h_n}) = \liminf_{n \to +\infty} E_{h_n}(u_{h_n}) < +\infty.
\end{equation}
We define the Radon measures $\lambda_{h_n}^{ij}$ by setting

\begin{equation*}
\lambda_{h_n}^{ij}(W) := \frac{1}{\sqrt{h_n}}\sigma_{ij}\int_W u_{h_n}^i e^{-h\Delta_{\xi}}u_{h_n}^j d\mu,\ W \in \mathcal{B}(M).
\end{equation*}
Then by~\eqref{eq:limisliminf}, upon passing to a further subsequence, we may assume that there exist Radon measures $\lambda^{ij}$ such that

\begin{equation}\label{eq:weakconvstar}
\lim_{n \to +\infty} \lambda_{h_n}^{ij} = \lambda^{ij}\ \text{weakly-$*$ in the sense of Radon measures}.
\end{equation}
In particular we obtain that

\begin{align*}
\liminf_{n \to +\infty} E_{h_n}(u_{h_n}) = \liminf_{n \to +\infty} \sum_{i,j} \sigma_{ij}\lambda_{h_n}^{ij}(M) \ge  \sum_{i,j} \sigma_{ij}\lambda^{ij}(M) .
\end{align*}
Thus to conclude the proof of the $\Gamma$-$\liminf$ inequality it suffices to show the following:

It holds that if $\underline{x} \in \Sigma_{ij}$ then

\begin{equation}\label{eq:claimondensity}
\sum_{m} \sigma_{mq}\frac{d\lambda^{mq}}{{d|Du^i|_{\xi}}}(\underline{x}) \ge \frac{2\sigma_{ij}}{\sqrt{\pi}}.
\end{equation}
Indeed, if~\eqref{eq:claimondensity} is true, then using the fact that the interfaces $\Sigma_{ij} = \partial^* \Omega_i \cap \partial^* \Omega_j$ are disjoint

\begin{align*}
\left(\sum_{m, q} \sigma_{mq}\lambda^{mq}\right)(M) &\ge \sum_{i < j} \left(\sum_{m, q} \sigma_{mq}\lambda^{mq}\right)(\Sigma_{ij})
\\ &\ge \sum_{i < j} \int_{\Sigma_{ij}} \sum_{m, q} \sigma_{mq}\frac{d\lambda^{mq}}{d|Du^i|_{\xi}}d|Du^i|_{\xi}
\\ &\ge \sum_{i < j} \frac{2\sigma_{ij}}{\sqrt{\pi}}\int_{\Sigma_{ij}}d|Du^i|_{\xi}
\\ &= \frac{1}{\sqrt{\pi}}\sum_{i,j}\sigma_{ij}|Du^i|_{\xi}(\Sigma_{ij}).
\end{align*}

We now prove~\eqref{eq:claimondensity}. Fix $\delta > 0$, then there exists $R < \frac{\operatorname{inj}(M)}{2}$ such that for any $x \in M$

\begin{equation}\label{eq:lowerlipdelta}
y,z \in B_{\frac{R}{2}}(x) \Rightarrow  d(y,z) \le (1+\delta)|\operatorname{exp}_x^{-1}(y) - \operatorname{exp}_x^{-1}(z)|.
\end{equation}
Fix $i, j \in \{1, . . . , P\}$, with $i \neq j$ and $\underline{x} \in \Sigma_{ij}$. For every $m, q \in \{1, . . . , P\}$ with $m \neq q$ we have that

\begin{equation*}
\frac{d\lambda^{mq}}{d|Du^i|_{\xi}}(\underline{x}) = \lim_{r \downarrow 0} \frac{\lambda^{mq}(B_r(\underline{x}))}{|Du^i|_{\xi}(B_r(\underline{x}))}.
\end{equation*}
Observe also that, using Lemma~\ref{lemma:chart_fp} applied with $V = B_r(\underline{x})$ and $\psi(y) = \operatorname{exp}_{\underline{x}}^{-1}(y)$,

\begin{equation*}
\lim_{r\downarrow 0} \frac{|Du^i|_{\xi}(B_r(\underline{x}))}{\omega_{k-1}r^{k-1}\gamma(\underline{o})} = \lim_{r \downarrow 0} \frac{\int_{B_r(\underline{o})}\gamma d\mathcal{H}^{k-1}}{\omega_{k-1}r^{k-1}\gamma(\underline{o})} = 1.
\end{equation*}
In particular

\begin{equation*}
\frac{d\lambda^{mq}}{d|Du^i|_{\xi}}(\underline{x}) = \lim_{r \downarrow 0} \frac{\lambda^{mq}(B_r(\underline{x}))}{\omega_{k-1}r^{k-1}\gamma(\underline{o}) }.
\end{equation*}
Observe that  there exists an at most countable set $Q \subset \mathbf{R}$ such that if $ r \not \in Q$

\begin{equation*}
\lambda^{mq}(\partial B_r(\underline{x})) = 0.
\end{equation*}
Thus, by the weak convergence~\eqref{eq:weakconvstar} of the $\lambda_{h_n}^{mq}$ we have

\begin{equation*}
\frac{d\lambda^{mq}}{d|Du^i|_{\xi}}(\underline{x}) = \lim_{r \downarrow 0, r \not \in Q} \lim_{n \to +\infty} \frac{\lambda_{h_n}^{mq}(B_r(\underline{x}))}{\gamma(\underline{o}) \omega_{k-1}r^{k-1}}.
\end{equation*}
We now set $\tilde{u}_{h_n} = u_{h_n} \circ \operatorname{exp}_{\underline{x}}$. Given a measurable function $f$ defined on $B_r(\underline{o})$ we define the blow-up at scale $r$ as $R_rf(y) := f(ry),\ y \in B_1$. By De Giorgi's structure theorem we know that

\begin{align}
&\lim_{r \downarrow 0} R_r \tilde{u}^i = \chi_{H_{\nu^{(i)}}}\ \text{in}\ L^1(B_1),\label{eq:blowupi}
\\ &\lim_{r \downarrow 0} R_r \tilde{u}^j = \chi_{H_{\nu^{(j)}}}\ \text{in}\ L^1(B_1)\label{eq:blowupj}
\end{align}
where we define

\begin{equation*}
H_{\nu^{(m)}} := \{ z \in \mathbf{R}^k:\ z \cdot \nu^{(m)} \le 0\},\ m \in \{1, . . . , P\}.
\end{equation*}
Here $\nu^{(m)}$ is the outer unit normal of $\operatorname{exp}_{\underline{x}}^{-1}(\Omega_m) \subset \mathbf{R}^k$ at $\underline{o}$. Observe furthermore that for $q \neq i, j$ it holds that

\begin{equation*}
\lim_{r \downarrow 0} R_r \tilde{u}^q = 0\ \text{in}\ L^1(B_1).
\end{equation*}
Indeed, this follows by the constraint $\sum_m R_{r}u^m_{h_n} = 1$ and~\eqref{eq:blowupi},~\eqref{eq:blowupj}. Upon selecting a subsequence, we may thus choose a sequence $r_n$ of radii such that

\begin{align*}
&\lim_{n \to +\infty} r_n = \lim_{n \to +\infty} \frac{h_n}{r_n^2} = 0,
\\ &\lim_{n \to +\infty} \frac{\lambda^{mq}_{h_n}(B_{r_n}(\underline{x}))}{\omega_{k-1}r_n^{k-1}\gamma(\underline{o})} = \frac{d\lambda^{mq}}{d|Du^i|}(\underline{x}),
\\ &\lim_{n \to +\infty} R_{r_n}\tilde{u}^i_{h_n} = \chi_{H_{\nu^{(i)}}}\ \text{in}\ L^1(B_1),
\\ &\lim_{n \to +\infty} R_{r_n}\tilde{u}^j_{h_n} = \chi_{H_{\nu^{(j)}}}\ \text{in}\ L^1(B_1),
\\ &\lim_{n \to +\infty} R_{r_n}\tilde{u}^m_{h_n} = 0\ \text{in}\ L^1(B_1)\ \text{for}\ m \neq i, j.
\end{align*}
We now use once more the expansion~\eqref{eq:asym_exp} with some $N \ge \frac{k}{2}$ and observe that

\begin{align*}
&\bigg| \lambda_{h_n}^{mq}(B_{r_n}(\underline{x})) - \sum_{l=0}^N\frac{1}{\sqrt{h_n}}\int_{B_{r_n}(\underline{x})}u_{h_n}^m\int_{B_{r_n}(\underline{x})}\frac{e^{\frac{-d^2(x,y)}{4h_n}}}{(4\pi h_n)^{k/2}}v_l(x,y)h_n^lu_{h_n}^q(y)d\mu(y)d\mu(x) \bigg|
\\ &\le C\sqrt{h_n}r_n^{2k}.
\end{align*}
Moreover, similarly as for~\eqref{eq:boiundlgeone} we get that for $l \ge 1$

\begin{align*}
\left|\frac{1}{\sqrt{h_n}}\int_{B_{r_n}(\underline{x})}u_{h_n}^m\int_{B_{r_n}(\underline{x})}\frac{e^{\frac{-d^2(x,y)}{4h_n}}}{(4\pi h_n)^{k/2}}{v_l(x,y)h_n^lu_{h_n}^q(y)d\mu(y)d\mu(x)} \right|  \le C\sqrt{h_n}r_n^k.
\end{align*}
From these two estimates we conclude that

\begin{align*}
&\begin{aligned}
\sum_{m,q} \sigma_{mq}\frac{d\lambda^{mq}}{d|Du^i|_{\xi}}(\underline{x}) 
\end{aligned}
\\ &\begin{aligned}
= \lim_{n \to +\infty} \sum_{m,q} \sigma_{mq}&\frac{1}{\omega_{k-1}r_n^{k-1}\gamma(\underline{o})\sqrt{h_n}}\times
\\ &\int_{B_{r_n}(\underline{x})}u_{h_n}^m\int_{B_{r_n}(\underline{x})}\frac{e^{\frac{-d^2(x,y)}{4h_n}}}{(4\pi h_n)^{k/2}}{v_0(x,y)u_{h_n}^q(y)d\mu(y)d\mu(x)}.
\end{aligned}
\end{align*}
By~\eqref{eq:lowerlipdelta}, for $n$ large enough the previous limit may be estimated from below by

\begin{equation}\label{eq:toesteuclid}
\liminf_{n \to +\infty} \frac{c_n}{\gamma(\underline{o})\omega_{k-1}r_n^{k-1}\sqrt{h_n}}\sum_{m,q} \sigma_{mq}\int_{B_{r_n}(\underline{o})}\tilde{u}_{h_n}^m\int_{B_{r_n}(\underline{o})}\frac{e^{\frac{-(1+\delta)^2|x-y|^2}{4h_n}}}{(4\pi h_n)^{k/2}}\tilde{u}^q_{h_n}dy dx,
\end{equation}
where $\tilde{u} := u \circ \operatorname{exp}_{\underline{x}}$ and

\begin{equation*}
c_n := \inf_{x, y \in B_{r_n}(\underline{o})} \left\{ v_0(\operatorname{exp}_{\underline{x}}(x), \operatorname{exp}_{\underline{x}}(y))\gamma(x) \gamma(y) \right\}.
\end{equation*}
Observe that $c_n \to \gamma(\underline{o})$ as $n \to +\infty$. In particular~\eqref{eq:toesteuclid} equals

\begin{equation*}
\liminf_{n \to +\infty} \frac{1}{\omega_{k-1}r_n^{k-1}\sqrt{h_n}}\sum_{m,q} \sigma_{mq}\int_{B_{r_n}(\underline{o})}\tilde{u}_{h_n}^m\int_{B_{r_n}(\underline{o})}\frac{e^{\frac{-(1+\delta)^2|x-y|^2)}{4h_n}}}{(4\pi h_n)^{k/2}}\tilde{u}^q_{h_n}dy dx.
\end{equation*}
We now perform the changes of variables $x \mapsto r_nx$ and $y \mapsto r_ny$, so that the previous quantity is equal to

\begin{equation}\label{eq:beforegammaliminf}
\liminf_{n \to +\infty} \frac{1}{\omega_{k-1}(1+\delta)^{k+1}}E^{B_1}_{\frac{h_n}{r_n^2(1+\delta)^2}}(R_{r_n}\tilde{u}_{h_n}\mathbf{1}_{B_{1}}),
\end{equation}
where we define for $t>0$ and $f \in \mathcal{A}_{B_1} := \left\{ f: B_1 \to [0,1]^P\ \text{such that}\ \sum_m f^m = 1\right\}$

\begin{equation*}
E^{B_1}_{t}(f) := \sum_{mq} \sigma_{mq} \frac{1}{\sqrt{t}} \int_{B_1} f^m G_t *f^q dx.
\end{equation*}
Here $G_t$ denotes the standard $k$-dimensional Euclidean heat kernel at time $t$. Let $\beta \in C^{\infty}_c(B_1)$, $0 \le \beta \le 1$, then for $f \in \mathcal{A}_{B_1}$

\begin{equation}\label{eq:localized_th_en}
E^{B_1}_t(f) \ge E^{B_1}_t(f, \beta) := \sum_{m, q} \sigma_{mq} \frac{1}{\sqrt{t}} \int_{B_1}\beta f^m G_t *f^q dx.
\end{equation}
We record the following result, a proof of which is given in the \hyperlink{sec:appendix}{Appendix}.

\begin{theorem}\label{thm:gammaconvlocalized}
If $\sigma \in \mathbf{R}^{k \times k}$ is symmetric, $\sigma_{mm} = 0$, $\sigma$ satisfy the triangle inequality and $\beta \in C^{\infty}_c(B_1)$ with $\beta \ge 0$, then on $\mathcal{A}_{B_1}$
\begin{equation*}
\Gamma - \lim_{t \downarrow 0} E^{B_1}_t(\cdot, \beta) = E(\cdot, \beta)\ \text{in}\ L^1(B_1),
\end{equation*}
where we define, for $f \in \mathcal{A}_{B_1}$,
\begin{equation*}
E(u, \beta) := \begin{cases}
\frac{1}{\sqrt{\pi}}\sum_{m,q} \sigma_{mq}\int_{S_{mq}} \beta(x) d\mathcal{H}^{d-1}(x)\ &\text{if}\ f \in BV(B_1, \{0,1\}^P),
\\ +\infty\ &\text{otherwise}.
\end{cases}
\end{equation*}
Here, for $f \in BV(B_1, \{0,1\}^P)$, we set $S_{mp} := \partial^*\{f^m = 1\} \cap \partial^*\{f^q = 1\}$.
\end{theorem}

In particular, we may use the $\Gamma$-$\liminf$ part of Theorem~\ref{thm:gammaconvlocalized} in~\eqref{eq:beforegammaliminf} to obtain that for any $\beta \in C^{\infty}_c(B_1), 0 \le \beta \le 1$ we have

\begin{align*}
\sum_{m,q} \sigma_{mq}\frac{d\lambda^{mq}}{{d|Du^i|_{\xi}}(\underline{x})} \ge \frac{1}{\sqrt{\pi}\omega_{k-1}(1+\delta)^{k+1}}\sigma_{ij}\bigg(&\int_{\{\nu^{(i)} \cdot x = 0\}} \beta(x) d\mathcal{H}^{d-1}(x) 
\\ &+\int_{\{\nu^{(j)} \cdot x = 0\}} \beta(x) d\mathcal{H}^{d-1}(x) \bigg).
\end{align*}
Taking the supremum over all such $\beta$ gives

\begin{equation*}
\sum_{m,q} \sigma_{mq}\frac{d\lambda^{mq}}{{d|Du^i|_{\xi}}(\underline{x})} \ge \frac{2\sigma_{ij}}{\sqrt{\pi}\omega_{k-1}(1+\delta)^{k+1}}\mathcal{H}^{k-1}(B_1^{(k-1)}) = \frac{2\sigma_{ij}}{\sqrt{\pi}(1+\delta)^{k+1}}.
\end{equation*}
The previous inequality holds for every $\delta>0$, thus if we let $\delta \downarrow 0$ we recover~\eqref{eq:claimondensity} and the proof of the $\Gamma$-$\limsup$ inequality is completed.

Compactness. To prove the last item of the theorem we proceed adapting the ideas of~\cite{Esedoglu2015} for the flat case. Fix $i \in \{1, . . . , P\}$ and define $m := \min_{j\neq i} \sigma_{ij}$. Then if $u \in \mathcal{M}$ we have that

\begin{align}
E_h(u) &= \sum_{i,j} \sigma_{ij} \int_M u^j e^{-h\Delta_\xi}u^i d\mu \nonumber
\\ &\ge m \int_M (1-u^i) e^{-h\Delta_{\xi}}u^i d\mu\nonumber
\\ &= \frac{m}{2} \int_{M \times M} p(h, x, y)((1-u^i(x))u^i(y) + u^i(x)(1-u^i(y))d\mu(y)d\mu(x)\label{eq:bound_thresh}
\\ &\ge  \frac{m}{2} \int_{M \times M} p(h, x, y)|u^i(y) - u^i(x)|d\mu(y)d\mu(x).\nonumber
\end{align}
We now fix $C>0$ to be determined later. By Stokes theorem $\int_M \nabla_x p(h,x,y) d\mu(y) = 0$, thus using the Gaussian upper bound~\eqref{eq:gauss_up_bds_II} and the Gaussian lower bound~\eqref{eq:gauss_up_bds_I} we observe

\begin{align*}
\int_M |De^{-Ch\Delta_{\xi}}u^i|_{\xi} &= \int_M |\nabla e^{-Ch\Delta_{\xi}}u^i|_x d\mu(x)
\\ &= \int_M \left| \int_M \nabla_x p(h, x, y) u^i(y) \right|_x d\mu(x)
\\ &\le \int_{M \times M} |\nabla_x p(h, x, y)|_x |u^i(y) - u^i(x)| d\mu(y)d\mu(x)
\\ &\le \frac{C_1}{Q_1}\int_{M \times M}\frac{\mu(B_{\sqrt{C_2Ch}}(x))}{\mu(B_{\sqrt{Ch}})}p\left(\frac{C_2Ch}{Q_2}, x, y\right)|u^i(y) - u^i(x)|d\mu(y)d\mu(x).
\end{align*}
If we take $C = \frac{Q_2}{C_2}$, using the doubling property~\eqref{eq:doubling} and the bound~\eqref{eq:bound_thresh} we end up with

\begin{align*}
\int_M |De^{-Ch\Delta_{\xi}}u^i|_{\xi} &\le \frac{C_1}{Q_1} \int_{M\times M} \frac{\mu(B_{\sqrt{Q_2h}}(x))}{\mu(B_{\sqrt{\frac{Q_2h}{C_2}}}(x))}p(h, x, y) |u^i(y) - u^i(x)|d\mu(y)d\mu(x)
\\ &\le \tilde{C} E_h(u).
\end{align*}
In particular, using this with $u = u_h$ we see that for every $i \in \{1,  . . .  ,P\}$

\begin{equation*}
\sup_{h > 0} \int_M |De^{-Ch\Delta_{\xi}}u_h^i|_{\xi} < +\infty.
\end{equation*}
By Lemma~\ref{lem:compactness_bv} we know that up to extracting a subsequence, for every $i \in \{1,  .  . . , P\}$ there exists $v^i \in BV(M)$ such that

\begin{equation*}
\lim_{h \downarrow 0} \Vert e^{-Ch\Delta_{\xi}}u_h^i - v^i\Vert_{L^1(M)} = 0.
\end{equation*}
The result now follows by observing that

\begin{equation*}
\lim_{h \downarrow 0} \Vert e^{-Ch\Delta_{\xi}}u_h^i - u^i_h \Vert_{L^1(M)} = 0.\qedhere
\end{equation*}
\end{proof}

\section{Appendix}\hypertarget{sec:appendix}{}

\subsection*{Proof of Theorem~\ref{thm:gamma_dirichlet}}

The proof is a slight modification of~\cite[Theorem 1.4]{GarciaTrillos2018}. The idea is still to reduce the problem to a nonlocal $\Gamma$-convergence result, namely to reduce it to the following statement.
\begin{theorem}\label{thm:nonlocal_local_dirichlet}
Let $M$ be a $k$-dimensional compact Riemannian submanifold of $\mathbf{R}^d$. Assume that $\eta$ is as in Section~\ref{sec:mainres}, let $\xi > 0$ be a smooth function on $M$. Given $\epsilon > 0$ and $u \in L^2(M)$ define

\begin{equation}
G_{\epsilon}(u) := \frac{1}{\epsilon^2} \int_{M \times M} \frac{1}{\epsilon^k}\eta\left( \frac{|x-y|_{d}}{\epsilon}\right)(u(x) - u(y))^2\xi(x)\xi(y)d\volm(x)d\volm(y),\nonumber
\end{equation}
where $|\cdot |_d$ denotes the Euclidean distance in $\mathbf{R}^d$. Then

\begin{equation}
\Gamma - \lim_{\epsilon \downarrow 0}G_{\epsilon} = 2C_2E,\nonumber
\end{equation}
where $E$ is the Dirichlet energy~\eqref{eq:dirichlet_energy}. Moreover, for any $u \in C^{\infty}(M)$ we have that $\limsup_{\epsilon \downarrow 0} G_{\epsilon}(u) \le 2C_2E(u)$. Finally, we have the following compactness property: if $\epsilon_n \downarrow 0$ and $u_n$ are such that

\begin{equation}
\sup_{n \in \mathbf{N}} G_{\epsilon_n}(u_n) < +\infty,\ \sup_{n \in \mathbf{N}} \Vert u_n \Vert_{L^2(M)} < +\infty,\nonumber
\end{equation}
then the sequence $u_n$ is precompact in $L^2(M)$.
\end{theorem}
To reduce the proof of Theorem~\ref{thm:gamma_dirichlet} to Theorem~\ref{thm:nonlocal_local_dirichlet} one proceeds along the same lines of the proof of~\cite[Theorem 1.4]{GarciaTrillos2018}. Let $T_n$ be the optimal transport maps obtained by applying Theorem \ref{thm:transp_plan_existence}. To follow the proof of~\cite[Theorem 1.4]{GarciaTrillos2018} the only additional observation is that, since $M$ is a Riemannian submanifold of $\mathbf{R}^d$, for any $x, y \in M$ we have

\begin{equation}
|x-y|_d \le d_M(x, y).\nonumber
\end{equation}
In particular this yields

\begin{equation}
\Vert I - T_n \Vert_{\infty} \le \sup_{x \in M} d_M(x, T_n(x)).\nonumber
\end{equation}
This will give the reduction to Theorem~\ref{thm:gamma_dirichlet}. In the case $k = 2$ the previous argument works as long as one assumes
\begin{equation*}
\frac{\epsilon_n n^{1/2}}{\log^{3/4}(n)} \gg 1.
\end{equation*}
The extra logarithmic factor may be removed by using \cite[Proposition 2.11]{Calder2022}.

We are left with proving Theorem~\ref{thm:nonlocal_local_dirichlet}. This can in turn be deduced from the corresponding result in the Euclidean case, namely the following.
\begin{theorem}\label{thm:nonlocal_local_euclidean}
Let $D \subset \mathbf{R}^k$ be a bounded open set with smooth boundary, let $\tilde{\xi}: \overline{D} \to (0, +\infty)$ be a smooth function. Then
\begin{equation}
\Gamma-\lim_{\epsilon \downarrow 0} \tilde{G}^{D, \tilde{\xi}}_{\epsilon} = 2C_2\tilde{E}^{D, \tilde{\xi}}\ \text{in}\ L^2(D),\nonumber
\end{equation}
and $\lim_{\epsilon \downarrow 0} \tilde{G}^{D, \tilde{\xi}}_{\epsilon}(u) = 2C_2\tilde{E}^{D, \tilde{\xi}}(u)$ whenever $u \in L^2(\overline{D})$, where we set for $u \in L^2(D)$

\begin{equation}
\tilde{G}^{D, \tilde{\xi}}_{\epsilon}(u) = \frac{1}{\epsilon^{k+2}} \int_{D\times D}\eta\left(\frac{|x-y|_k}{\epsilon}\right)|u(x) - u(y)|^2 \tilde{\xi}(x)\tilde{\xi}(y) dx dy,\nonumber
\end{equation}
and
\begin{equation}
\tilde{E}^{D, \tilde{\xi}}(u) = \begin{cases}
\frac{1}{2}\int_{D} |Du|^2\tilde{\xi}^2 dx\ &\text{if}\ u \in H^1(D)
\\ +\infty\ &\text{otherwise}.
\end{cases}\nonumber
\end{equation}
Finally, the following compactness property holds true: if $\epsilon_n \downarrow 0$ and $u_n$ are such that

\begin{equation}
\sup_{n \in \mathbf{N}} \tilde{G}^{D, \tilde{\xi}}_{\epsilon_n}(u_n) < +\infty,\ \sup_{n \in \mathbf{N}} \Vert u_n \Vert_{L^2(D)} < +\infty\nonumber
\end{equation}
then the sequence $u_n$ is precompact in $L^2(D)$.
\end{theorem}
With this result, we can prove Theorem~\ref{thm:nonlocal_local_dirichlet}. 
\begin{proof}[Proof of Theorem~\ref{thm:nonlocal_local_dirichlet}]
$\Gamma$-$\liminf$. Let $u_n \to u$ in $L^2(M)$. We want to prove that for any sequence $\epsilon_n \downarrow 0$ we have

\begin{equation}\label{eq:claim_liminf}
\liminf_{n \to +\infty} G_{\epsilon_n}(u_n) \ge 2C_2E(u).
\end{equation}
To this aim, we may assume as before that the left hand side of~\eqref{eq:claim_liminf} is finite. For any given $R > 0$ the family

\begin{equation}
\mathcal{F}_R := \left\{ \overline{B_{r}(x)} \subset M:\ r \le R,\ B_{r}(x) \subset \subset \Psi(U),\ \Psi\ 1-\text{Lipschitz chart}\right\}\nonumber
\end{equation}
is a Vitali covering. Since the manifold is compact, or more precisely by the doubling property~\eqref{eq:doubling}, we can select countably many disjoint balls $B_i$  in the above family, such that

\begin{equation}\label{eq:covering_vol}
\volm\left(M \setminus \bigcup_{i=1}^\infty \overline{B_i}\right) = 0.
\end{equation}
By construction, each ball $B_i$ is contained in $\Psi_i(U_i)$ for some $1$-Lipschitz local parametrization $\Psi_i$. Here $1$-Lipschitz is understood between Euclidean spaces, i.e.
\begin{equation}\label{eq:1_lip}
|\Psi_i(y_1) - \Psi_i(y_2)|_d \le |y_1 - y_2|_k,\ y_1, y_2 \in U_i. 
\end{equation}
In particular we have that since $\eta$ is non-increasing
\begin{equation}\label{eq:low_bd_eta}
\eta\left( \frac{|\Psi_i(y_1) - \Psi_i(y_2)|_d}{\epsilon_n} \right) \ge \eta\left( \frac{|y_1 - y_2|_k}{\epsilon_n} \right).
\end{equation}
By~\eqref{eq:low_bd_eta} and since the balls are disjoint we obtain
\begin{align*}
&G_{\epsilon_n}(u_n) 
\\ &\ge \sum_{i \in \mathbf{N}} \frac{1}{\epsilon^2_n} \int_{U_i \times U_i}\frac{1}{\epsilon_n^k}\eta\left( \frac{|x-y|_k}{\epsilon}\right)(u_n\circ \Psi_i(x) - u_n \circ \Psi(y))^2 \tilde{\xi}_i(x) \tilde{\xi}_i(y)dx dy\nonumber
\\ &= \sum_{i \in \mathbf{N}} \tilde{G}^{U_i, \tilde{\xi}_i}_{\epsilon_n}(u_n \circ \Psi_i),\nonumber
\end{align*}
where $\tilde{\xi}_i(y) = \xi(\Psi_i(y))\sqrt{\operatorname{det}(g)}$. Now an application of Fatou's Lemma, Theorem~\ref{thm:nonlocal_local_euclidean} and~\eqref{eq:covering_vol} give the $\liminf$ inequality.

$\Gamma$-$\limsup$. By a diagonal argument, we may reduce to proving the $\Gamma$-$\limsup$ inequality in the case when $u \in H^1(M) \cap C^{\infty}(M)$. For such a function and any sequence $\epsilon_n \downarrow 0$ we claim that

\begin{equation}
\limsup_{n \to +\infty} G_{\epsilon_n}(u) \le 2C_2E(u).\nonumber
\end{equation}

We will in a second moment construct a suitable covering $\{W_1, . . . , W_N\}$ which satisfies

\begin{equation}
M \subset \bigcup_{i=1}^N W_i,\nonumber
\end{equation}
where $N \in \mathbf{N}$ and $W_i = \Psi_i(U_i)$ with $\Psi_i$ local parametrizations defined on a bounded domain $U_i \subset \mathbf{R}^k$ with Lipschitz boundary such that
\begin{equation}\label{eq:lip_low}
|\Psi_i(y_1) - \Psi_i(y_2)|_d \ge |y_1 - y_2|_k\ \text{for}\ y_1, y_2 \in U_i.
\end{equation}
Let $\delta$ be the Lebesgue number of the given covering. Define the set 
\begin{equation}
F_{\delta} := \left\{ (x, y) \in M \times M:\ d(x, y) \ge \delta \right\}.\nonumber
\end{equation}
It is clear that $F_{\delta}$ is compact, thus by continuity we have that

\begin{equation}
|x-y|_d \ge c_\delta > 0,\ x, y \in F_{\delta}.\nonumber
\end{equation}
In particular by the exponential decay of $\eta$ we get that there exists two positive constants $c_1, c_2$ such that

\begin{equation}\label{eq:eta_f_delta}
\eta\left( \frac{|x-y|_d}{\epsilon}\right) \le c_1 \operatorname{exp}\left(\frac{-c_2c_\delta}{\epsilon} \right),\ x, y \in F_\delta.
\end{equation}
Now observe that 

\begin{align*}
G_{\epsilon_n}(u) = &\frac{1}{\epsilon_n^{2+k}}\int_{F_{\delta}}\eta\left( \frac{|x-y|_d}{\epsilon_n} \right) |u(x) - u(y)|^2\xi(x) \xi(y)d\volm(x)d\volm(y)
\\ &+\frac{1}{\epsilon_n^{2+k}}\int_{M \times M \setminus F_{\delta}}\eta\left( \frac{|x-y|_d}{\epsilon_n} \right) |u(x) - u(y)|^2\xi(x) \xi(y)d\volm(x)d\volm(y).
\end{align*}
Recalling~\eqref{eq:eta_f_delta} we may observe that the first term on the right hand side converges to zero as $n \to +\infty$. We claim that the $\limsup$ of the second right hand side term is bounded above by

\begin{equation}
\sum_{i=1}^N \int_{W_i} |\nabla u|^2 \xi d\volm.\nonumber
\end{equation}
To show this, observe that if $(x,y) \in M \times M \setminus F_{\delta}$ then by definition there exists $1\le i \le N$ such that $(x, y) \in W_{i} \times W_i$, in particular

\begin{align*}
&\frac{1}{\epsilon_n^{2+k}}\int_{M \times M \setminus F_{\delta}}\eta\left( \frac{|x-y|_d}{\epsilon_n} \right) |u(x) - u(y)|^2\xi(x) \xi(y)d\volm(x)d\volm(y)
\\ &\le \sum_{i=1}^N \frac{1}{\epsilon_n^{2+k}}\int_{W_i \times W_i}\eta\left( \frac{|x-y|_d}{\epsilon_n} \right) |u(x) - u(y)|^2\xi(x) \xi(y)d\volm(x)d\volm(y)
\\ &=\sum_{i=1}^N \tilde{G}_{\epsilon_n}^{U_i, \tilde{\xi}_i}(u \circ \Psi_i),
\end{align*}
where $\tilde{\xi}_i(y) = \xi(\Psi_i(y))\sqrt{\operatorname{det}(g)}$. Recalling Theorem~\ref{thm:nonlocal_local_euclidean}, if we let $n \to +\infty$ we obtain

\begin{equation}
\limsup_{n \to +\infty} G_{\epsilon_n}(u) \le C_2\sum_{i=1}^N \int_{W_i}|\nabla u|^2\xi d\volm.\nonumber
\end{equation}

We now claim that given any $\alpha > 0$ we can find $N \in \mathbf{N}$ and a covering $W_1, . . . , W_N$ as before such that 

\begin{equation}\label{eq:approx_cov_cl}
\sum_{i=1}^N \int_{W_i}|\nabla u|^2\xi d\volm - \int_M |\nabla u|^2 \xi d\volm < \alpha.
\end{equation}
This can be done as follows. Given any point $x \in M$ we can find a smooth function $\gamma: \mathbf{R}^k \to \mathbf{R}^{d-k}$ and a number $R>0$ such that, upon translating and rotating the axes, the map 

\begin{equation}
\Psi(y) = (y, \gamma(y)),\ y \in Q(0,R) := (0, R)^k\nonumber
\end{equation}
is a local parametrization around $x$. Clearly we have that~\eqref{eq:lip_low} is true. Define $V_x := \Psi(Q(0, \frac{R}{2}))$. Since the manifold is compact, we can find $N \in \mathbf{N}$ and points $x_1, . . . , x_N$ such that the sets $V_i := V_{x_i}$, $i=1, . . . , N$, cover $M$. Now define $\tilde{V}_1 = V_1$, $\tilde{V}_{i+1} = V_{i+1} \setminus \cup_{j=1}^i V_i$. Then $\{ \tilde{V}_i\}$ is a partition of $M$. Define $A_i := \Psi_i^{-1}(V_i)$. Then the sets $A_i \subset Q(0, \frac{R_i}{2})$ have Lipschitz boundary. Given $\theta > 0$ sufficiently small define, for any $1 \le i \le N$

\begin{align*}
&A_i^{\theta} := \left\{ y \in Q(0, R_i):\ d(y, A_i) < \theta \right\},
\\ & W_i := \Psi_i(A_i^\theta).
\end{align*}
Clearly, $\{W_i, . . ., W_N\}$ is an open covering satisfying~\eqref{eq:lip_low}. We now check that it satisfies~\eqref{eq:approx_cov_cl} provided $\theta$ is small enough. Observe that there exists a constant $C>0$ such that for any $1 \le i \le N$

\begin{align*}
\volm(W_i \setminus \tilde{V_i}) &\le C\mathcal{L}^k\left( A_i^{\theta}\setminus A_i \right)
\\ & \le C \mathcal{L}^k\left( \left\{ y \in Q(0, R_i):\ |y - \partial A_i|_k \le \theta \right\} \right).
\end{align*}
Recall that, since $\partial A_i$ is $(k-1)$-rectifiable, we have

\begin{equation}
\mathcal{H}^{k-1}(\partial A_i) = \lim_{\theta \downarrow 0} \frac{\mathcal{L}^k\left( \left\{ y \in Q(0, R_i):\ |y - \partial A_i|_k \le \theta \right\} \right)}{\theta}.\nonumber
\end{equation}
The right-hand side is the Minkowski content, cf.~\cite[Theorem 3.2.39]{Federer1969}. In particular for a given $\tilde{\alpha} > 0$, we can choose $\theta$ so small that 

\begin{equation}
\volm(W_i \setminus \tilde{V_i})  \le C\tilde{\alpha}.\nonumber
\end{equation}
Now observe that since $\tilde{V}_i \subset W_i$

\begin{align*}
&\sum_{i=1}^N \int_{W_i} |\nabla u|^2\xi d\volm - \int_M |\nabla u|^2 \xi d\volm
\\ &=\sum_{i=1}^N \int_{W_i} |\nabla u|^2\xi d\volm - \int_{\tilde{V}_i} |\nabla u|^2 \xi d\volm
\\ &\le \tilde{C}N\tilde{\alpha}.
\end{align*}
Choosing $\tilde{\alpha} = \frac{\alpha}{\tilde{C}N}$ we get~\eqref{eq:approx_cov_cl}. In particular

\begin{equation}
\limsup_{n \to +\infty} G_{\epsilon_n}(u) \le 2C_2E(u) + 2C_2\alpha,\nonumber
\end{equation}
and letting $\alpha \downarrow 0$ we get the $\limsup$ inequality.

The compactness property follows easily from Theorem~\ref{thm:nonlocal_local_euclidean}.
\end{proof}

\subsection*{$\Gamma$-convergence of the localized thresholding energies}\label{subs:gamma_loc}

Here we sketch the proof of Theorem~\ref{thm:gammaconvlocalized}. The upper bound in the $\Gamma$-convergence is obtained by using Lemma 3.6 in~\cite{Laux2016}. For the lower bound, one just needs the following approximate monotonicity, which was proved by Otto and one of the authors in the first version of the preprint preceeding~\cite{Laux2016}, but did not appear in the published version. 

\begin{theorem}\label{thm:monotonicitylocalized}
Let $\sigma \in \mathbf{R}^{P \times P}$ be a symmetric matrix such that $\sigma_{ij}$ satisfy the triangle inequality. Let $\beta \in C^{\infty}_c(B_1)$, where $B_1 \subset \mathbf{R}^k$ is the unit ball. For $t>0$ define $E_{t}^{B_1}(\cdot, \beta)$ as in~\eqref{eq:localized_th_en}. Let $k_t(z) = \frac{1}{\sqrt{t}^k}k_1(\frac{z}{\sqrt{t}})$, with $k_1(z) = |z|G_1(z)$. Then, defining for $u: B_1 \to [0,1]^P$ with $\sum_m u^m = 1$,

\begin{equation*}
\tilde{E}_t(u) = \frac{1}{\sqrt{t}} \sum_{i,j} \sigma_{ij} \int u^i k_t * u^j dx,
\end{equation*}
we have that for all such $u$ and all $0 < h \le h_0$

\begin{align}\label{eq:monot_loc_stat}
E_{h_0}^{B_1}(u, \beta) \le \left( \frac{\sqrt{h_0} + \sqrt{h}}{\sqrt{h_0}} \right)^{k+1} E^{B_1}_{h}(u, \beta) + C\Vert D\beta \Vert_{L^{\infty}}\tilde{E}_h(u)\sqrt{h_0}.
\end{align}
Here $C$ is a constant that does not depend on $h$ nor on $h_0$.
\end{theorem}

The original proof was based on the ideas used for proving the monotonicity of the non-localized thresholding energies in~\cite{Esedoglu2015}. For the convenience of the reader, we include a proof for the simpler two phase setting. In that case one has to prove~\eqref{eq:monot_loc_stat} with

\begin{equation}
\tilde{E}^{B_1}_t(u, \beta) = \frac{1}{\sqrt{t}}\int_{B_1} \beta (1-u) G_t * u dx,\ u:B_1 \to [0,1]
\end{equation}
and
\begin{equation}
\tilde{E}_t(u) = \frac{1}{\sqrt{t}} \int_{B_1} k_t *(1-u) u dx,\ u :B_1 \to [0,1].
\end{equation}

\begin{proof}[Proof of Theorem~\ref{thm:monotonicitylocalized} in the two phase setting]
Clearly, statement~\eqref{eq:monot_loc_stat} is a consequence of the following two items.
\begin{align}
\sqrt{h_1}^{k+1} E_{h_1}^{B_1}(u, \beta) &\le \sqrt{h_2}^{k+1} E_{h_2}^{B_1}(u, \beta)\ &\forall 0 < h_1 \le h_2,\label{eq:logarithmic_mon}
\\ E_{N^2h}^{B_1}(u, \beta) &\le E_h^{B_1}(u, \beta) + C(N-1)\sqrt{h}\Vert D\beta \Vert_{\infty} \tilde{E}_h(u)\ &\forall N \in \mathbf{N},\ \forall h > 0.\label{eq:disc_mon}
\end{align}
To see this, let $0 < h \le h_0$. Let $N \in \mathbf{N}$ be such that 
\begin{equation*}
(N-1)\sqrt{h} \le \sqrt{h_0} < N\sqrt{h}.
\end{equation*}
Then we have
\begin{align*}
E_{h_0}^{B_1}(u, \beta) &\stackrel{{\eqref{eq:logarithmic_mon}}}{\leq} \left( \frac{N\sqrt{h}}{\sqrt{h_0}} \right)^{k+1}E_{N^2h}^{B_1}(u, \beta)
\\ &\stackrel{\eqref{eq:disc_mon}}{\leq}\left( \frac{N\sqrt{h}}{\sqrt{h_0}} \right)^{k+1}\left( E_h^{B_1}(u,\beta) + C(N-1)\sqrt{h}\Vert D\beta \Vert_{\infty} \tilde{E}_h(u)\right)
\\ &\stackrel{\phantom{\eqref{}}}{\leq}\left( \frac{\sqrt{h} + \sqrt{h_0}}{\sqrt{h_0}} \right)^{k+1} E_h^{B_1}(u,\beta) + C\sqrt{h_0}\Vert D\beta \Vert_{\infty} \tilde{E}_h(u).
\end{align*}
We are thus left with proving~\eqref{eq:logarithmic_mon} and~\eqref{eq:disc_mon}.
\\
\textit{Item~\eqref{eq:logarithmic_mon}}. This follows by showing that

\begin{equation*}
\frac{d}{d\sqrt{h}}\left( \sqrt{h}^{k+1}E_h^{B_1}(u, \beta)\right) \ge 0,
\end{equation*}
which follows by differentiation. Indeed, 
\begin{align*}
\frac{d}{d\sqrt{h}}\left( \sqrt{h}^{k+1}E_h^{B_1}(u, \beta)\right)  &= \frac{d}{d\sqrt{h}} \int_{B_1}\beta (1-u) G_1\left( \frac{z}{\sqrt{h}}\right) u dx
\\ &= -\frac{1}{\sqrt{h}} \int_{B_1}\beta (1-u) \nabla G_1\left( \frac{z}{\sqrt{h}} \right) \cdot \frac{z}{\sqrt{h}} u dx \ge 0.
\end{align*}
\textit{Item~\eqref{eq:disc_mon}}. We let $0 < h_0$ be such that $\sqrt{h_0} = \sqrt{h_1} + \sqrt{h}$. Then we observe that

\begin{align*}
\phantom{\sqrt{h_0}E_{h_0}^{B_1}(u, \beta)}
&\begin{aligned}\mathllap{\sqrt{h_0}E_{h_0}^{B_1}(u, \beta)} &= \int_{B_1} \int_{\mathbf{R}^k} \beta(x) (1-u)(x) G_1(z) u(x-\sqrt{h_1}z - \sqrt{h}z) dz dx
\end{aligned}
\\ & \begin{aligned}
\mathllap{} \le &\int_{B_1} \int_{\mathbf{R}^k} \beta(x) (1-u)(x-\sqrt{h_1}z) G_1(z) u(x-\sqrt{h_1}z - \sqrt{h}z) dz dx
\\ &+\int_{B_1} \int_{\mathbf{R}^k} \beta(x) (1-u)(x) G_1(z) u(x-\sqrt{h_1}z) dz dx,
\end{aligned}
\end{align*}
where we used the inequality

\begin{align*}
(1-u)u'' \le (1-u')u'' + (1-u)u',\ \forall u, u', u'' \in [0,1],
\end{align*}
applied to $u = u(x), u' = u(x-\sqrt{h_1}z), u'' = u(x-\sqrt{h}z - \sqrt{h_1}z)$. We record that the second term on the right hand side is equal to

\begin{align*}
\sqrt{h_1}E_{h_1}^{B_1}(u, \beta).
\end{align*}
The other term is estimated as follows. First we change variable in $x$, and then we estimate $|\beta(x) - \beta(x-\sqrt{h_1}z)| \le \Vert D\beta \Vert_{\infty} \sqrt{h_1} |z|$ to get
\begin{align*}
&\begin{aligned}
\int_{\mathbf{R}^k} \int_{B_1-\sqrt{h_1}z} \beta(x + \sqrt{h}z) (1-u)(x) G_1(z) u(x-\sqrt{h}z) dz dx
\end{aligned}
\\ &\begin{aligned}
\le  &\int_{\mathbf{R}^k} \int_{B_1-\sqrt{h_1}z} \beta(x) (1-u)(x) G_1(z) u(x-\sqrt{h}z) dx dz
\\ &+\sqrt{h_1}\Vert D\beta \Vert_{\infty} \int_{\mathbf{R}^k} \int_{B_1-\sqrt{h_1}z} (1-u)(x) |z|G_1(z) u(x-\sqrt{h}z) dz dx
\end{aligned}
\\ &\begin{aligned}
= &\int_{\mathbf{R}^k} \int_{B_1} \beta(x) (1-u)(x) G_1(z) u(x-\sqrt{h}z) dx dz
\\ &+\sqrt{h_1}\Vert D\beta \Vert_{\infty} \int_{\mathbf{R}^k} \int_{B_1-\sqrt{h_1}z} (1-u)(x) k_1(z) u(x-\sqrt{h}z) dz dx,
\end{aligned}
\end{align*}
where in the last equality we used the fact that $\beta$ is supported in $B_1$. Observe that

\begin{align*}
&\begin{aligned}
\int_{\mathbf{R}^k} \int_{B_1-\sqrt{h_1}z} (1-u)(x) k_1(z) u(x-\sqrt{h}z) dz dx
\end{aligned}
\\ &\begin{aligned}
\le\int_{\mathbf{R}^k} \int_{\mathbf{R}^k} (1-u)(x) k_1(z) u(x-\sqrt{h}z) dz dx
\end{aligned}
\\ &\begin{aligned}
= &\int_{\mathbf{R}^k}\int_{\mathbf{R}^k} (1-u)(x) k_{h}(z) u(x-z) dz dx
\end{aligned}
\\  &\begin{aligned}
=&\int_{\mathbf{R}^k}\int_{\mathbf{R}^k} (1-u)(x) k_{h}(z-x) u(z) dz dx
\end{aligned}
\\  &\begin{aligned}
=&\int_{\mathbf{R}^k}\int_{B_1} (1-u)(x) k_{h}(z-x) u(z) dz dx
\end{aligned}
\\  &\begin{aligned}
=&\int_{B_1}k_{h} * (1-u)(z) u(z) dz = \sqrt{h}\tilde{E}_{h}(u).
\end{aligned}
\end{align*}
Here we used that $u$ is supported in $B_1$. Putting things together we obtain that

\begin{align}\label{eq:mon_bef_induction}
\sqrt{h_0}E_{h_0}^{B_1}(u, \beta) \le &\sqrt{h_1}E_{h_1}^{B_1}(u, \beta) + \sqrt{h}E_h^{B_1}(u, \beta)
\\ &+\sqrt{h_1}\sqrt{h}\Vert D\beta \Vert_{\infty} \tilde{E}_{h}(u).\nonumber
\end{align}
If we now apply inductively~\eqref{eq:mon_bef_induction} with $h_1 = (N-1)^2h$ and $h_0 = N^2h$ one gets

\begin{align*}
N\sqrt{h} E_{N^2h}^{B_1}(u, \beta) &\le N\sqrt{h} E_h^{B_1}(u, \beta) + \sum_{i = 1}^{N-1} ih\Vert D\beta \Vert_{\infty} \tilde{E}_{h}(u)
\\ & =N\sqrt{h} E_h^{B_1}(u, \beta) + \frac{(N-1)N}{2}h\Vert D\beta \Vert_{\infty} \tilde{E}_{h}(u).
\end{align*}
Dividing by $N\sqrt{h}$ yields~\eqref{eq:disc_mon}.
\end{proof}

\section*{Data Availability}

The datasets generated during and/or analysed during the current study are available in the GitHub repository \url{https://github.com/jonalelmi/Data-th-en}.

\section*{Acknowledgements}

This project has received funding from the Deutsche Forschungsgemeinschaft (DFG, German Research Foundation) under Germany's Excellence Strategy -- EXC-2047/1 -- 390685813.

\bibliography{bib_first_draft}{}
\bibliographystyle{plain}

\end{document}